\documentclass[10pt]{article}
\usepackage{indentfirst,latexsym,bm}
\usepackage{amsfonts}
\usepackage{amssymb}
\usepackage[leqno]{amsmath}
\usepackage{amsbsy}
\usepackage{dsfont}
\usepackage{amsthm}
\usepackage{amscd}
\usepackage[all]{xy}
\usepackage{chemarrow}
\usepackage{multirow}
\begin{document}
\title{Some Hopf algebras of dimension $72$ without the Chevalley property}
\author{Naihong Hu\thanks{Email:\,nhhu@math.ecnu.edu.cn}\\{\small Department of Mathematics, East China Normal University,
Shanghai 200062, China}
\and Rongchuan Xiong\thanks{Email:\,52150601006@ecnu.cn}
\\{\small Department of Mathematics, East China Normal University,
Shanghai 200062, China}
}

\maketitle

\newtheorem{question}{Question}
\newtheorem{defi}{Definition}[section]
\newtheorem{conj}{Conjecture}
\newtheorem{thm}[defi]{Theorem}
\newtheorem{lem}[defi]{Lemma}
\newtheorem{pro}[defi]{Proposition}
\newtheorem{cor}[defi]{Corollary}
\newtheorem{rmk}[defi]{Remark}
\newtheorem{Example}{Example}[section]

\newcommand{\C}{\mathcal{C}}
\newcommand{\D}{\mathcal{D}}
\newcommand{\A}{\mathcal{A}}
\newcommand{\De}{\Delta}
\newcommand{\M}{\mathcal{M}}
\newcommand{\K}{\mathds{k}}
\newcommand{\E}{\mathcal{E}}
\newcommand{\Pp}{\mathcal{P}}
\newcommand{\Lam}{\lambda}
\newcommand{\As}{^{\ast}}
\newcommand{\Aa}{a^{\ast}}
\newcommand{\Ab}{(a^2)^{\ast}}
\newcommand{\Ac}{(a^3)^{\ast}}
\newcommand{\Ad}{(a^4)^{\ast}}
\newcommand{\Ae}{(a^5)^{\ast}}
\newcommand{\B}{b^{\ast}}
\newcommand{\BAa}{(ba)^{\ast}}
\newcommand{\BAb}{(ba^2)^{\ast}}
\newcommand{\BAc}{(ba^3)^{\ast}}
\newcommand{\BAd}{(ba^4)^{\ast}}
\newcommand{\BAe}{(ba^5)^{\ast}}

\newcommand{\CYD}{{}^{\C}_{\C}\mathcal{YD}}
\newcommand{\DM}{{}_{D}\mathcal{M}}
\newcommand{\BN}{\mathcal{B}}

\newcommand{\Ga}{g^{\ast}}
\newcommand{\Gb}{(g^2)^{\ast}}
\newcommand{\Gc}{(g^3)^{\ast}}
\newcommand{\Gd}{(g^4)^{\ast}}
\newcommand{\Ge}{(g^5)^{\ast}}
\newcommand{\X}{x^{\ast}}
\newcommand{\GXa}{(gx)^{\ast}}
\newcommand{\GXb}{(g^2x)^{\ast}}
\newcommand{\GXc}{(g^3x)^{\ast}}
\newcommand{\GXd}{(g^4x)^{\ast}}
\newcommand{\GXe}{(g^5x)^{\ast}}
\begin{abstract}
In this paper, we consider the Drinfeld double $\D$ of a $12$-dimensional Hopf algebra $\C$ over an algebraically
closed field of characteristic zero whose coradical is not a subalgebra and describe its simple modules,
projective covers of the simple modules and show that it is of wild representation type. Moreover,
we show that the Nichols algebras associated to non-simple indecomposable modules are infinite-dimensional.
In particular, for any object $V$ in $\CYD$, if $\BN(V)$ is finite-dimensional, then $V$ must be semisimple.
Finally, we describe the Nichols algebras associated to partial simple modules in terms of generators and relations.
As a byproduct, we obtain some Hopf algebras of dimension $72$ without the Chevalley property, that is, the coradical is not a subalgebra.

\end{abstract}

\section{Introduction}
Let $\K$ be an algebraically closed field of characteristic zero. In
1975, I. Kaplansky conjectured that every Hopf algebra over $\K$ of
prime dimension must be isomorphic to a group algebra which was
proved by Y. Zhu \cite{Z94} in 1994. Since then, more and more
mathematicians have been trying to classify finite-dimensional Hopf
algebras of a given dimension and have made some progress. As
the aforementioned, the classification of Hopf algebras of prime
dimension $p$ has been completed by Y. Zhu \cite{Z94} and all of
them are isomorphic to the cyclic group algebra of dimension $p$.
Further results have completed the classification of Hopf algebras
of dimension $p^2$ for $p$ a prime (see \cite{Ma96},
\cite{AS}, \cite{Ng02}), of dimension $2p$ for $p$ an odd prime
(see \cite{Ma95}, \cite{Ng05}), and of dimension $2p^2$ for $p$
an odd prime (see \cite{HN09}). For the statement of the
classification for dimensions up to $100$, we refer to \cite{BG13}
and the references therein. Although some progress has been made,
there are few general results and methods and even the
classification of Hopf algebras of some low dimensions is still
open. However, the classification of finite-dimensional pointed Hopf
algebras as a special class has made some astonishing breakthrough
(\cite{AS10}, et al) and the classification of those with abelian
groups as the coradicals is expected to be completed soon with the
help of the lifting method introduced by N. Andruskiewitsch and
H.-J.~Schneider \cite{AS98}.

Now we describe the lifting method briefly. Let $A$ be a finite
dimensional Hopf algebra such that the coradical $A_0$ is a Hopf
subalgebra, which implies that the coradical filtration
$\{A_n\}_{n=0}^{\infty}$ is a Hopf algebra filtration where
$A_n=A_0\bigwedge A_{n-1}$. Let $H=\text{gr}\; A$ be the associated graded
Hopf algebra, that is, $H=\oplus_{n=0}^{\infty}H(n)$ with $H(0)=A_0$
and $H(n)=A_n/A_{n-1}$. Denote by $\pi: H\rightarrow A_0$ the Hopf
algebra projection of $H$ onto $A_0=H(0)$, then $\pi$ splits the
inclusion $i: A_0\hookrightarrow H$ and thus by a theorem of Radford
\cite{R85}, $H\cong R\sharp A_0$, where $R=H^{co\pi}=\{h\in
H\mid (id\otimes \pi)\Delta_H(h)=h\otimes 1\}$ is a braided Hopf algebra
in ${}^{A_0}_{A_0}\mathcal{YD}$. Moreover,
$R=\oplus_{n=0}^{\infty}R(n)=R\cap H(n)$ with $R(0)=\K$ and
$R(1)=\Pp(R)$, the space of primitive elements of $R$, which is a
braided vector space called the infinitesimal braiding. In
particular, the subalgebra generated as a braided Hopf algebra by
$V$ is so-called Nichols algebra of $V$ denoted by $\BN(V)$, which
plays a key role in the classification of pointed Hopf algebra under
the following
\begin{conj} $($\text{\rm Conjecture\,2.7 \cite{AS02}}$)$
Any finite-dimensional braided Hopf algebra $R$ in ${}^{A_0}_{A_0}\mathcal{YD}$
satisfying $\Pp(R)=R(1)$ is generated by $R(1)$.
\end{conj}
Usually, if we fix $A_0$, then the lifting method consists of three steps:
\begin{itemize}
  \item Determine all braided vector spaces $V$ such that Nichols algebras $\BN(V)$ are finite-dimensional and
  describe $\BN(V)$ explicitly in terms of generators and relations.
  \item For such $V$, determine all possible finite-dimensional Hopf algebras $A$ such that the associated graded
  Hopf algebra $gr\,A$ is isomorphic to the bosonization $\BN(V)\sharp A_0$. We call $A$ a lifting of $V$ or $\BN(V)$ over $A_0$.
  \item Prove that any finite-dimensional Hopf algebra with $A_0$ as the coradical is generated by $V\sharp A_0\oplus A_0$.
\end{itemize}
So far, the lifting method has produced many striking results of the
classification of pointed or copointed Hopf algebras. For more
details about the results, we refer to \cite{A14}, \cite{BG13}
and the references therein.

 If $A$ is a Hopf algebra without the Chevalley property, then the coradical filtration $\{A_n\}_{n=0}^{\infty}$ is not a Hopf
 algebra filtration such that the associated graded coalgebra is not a Hopf algebra. To overcome this obstacle, Andruskiewitsch
 and Cuadra \cite{AC13} extended the lifting method by replacing the coradical filtration $\{A_n\}_{n=0}^{\infty}$ by the
 standard filtration $\{A_{[n]}\}_{n=0}^{\infty}$, which is defined recursively by
\begin{itemize}
  \item $A_{[0]}$ to be the subalgebra generated by the coradical $A_0$;
  \item $A_{[n]}=A_{[n-1]}\bigwedge A_{[0]}$.
\end{itemize}
Especially, if $A_0$ is a Hopf algebra, then $A_{[0]}=A_0$ and
standard filtration is just the coradical filtration. Under the
assumption that $S_A(A_{[0]})\subseteq A_{[0]}$, it turns out that
the standard filtration is a Hopf algebra filtration, and the
associated graded coalgebra
$S=gr\,A=\oplus_{n=0}^{\infty}A_{[n]}/A_{[n-1]}$ with $A_{[-1]}=0$
is also a Hopf algebra. If we denote as before, $\pi:S\rightarrow
A_0$ splits the inclusion $i:A_0\hookrightarrow S$ and thus by a
theorem of Radford, $S\cong R\sharp A_0$, where $R=S^{co\pi}=\{h\in
S\mid(id\otimes \pi)\Delta_S(h)=h\otimes 1\}$ is a braided Hopf algebra
in ${}^{A_0}_{A_0}\mathcal{YD}$. Moreover,
$R=\oplus_{n=0}^{\infty}R(n)=R\cap S(n)$ with $R(0)=\K$ and
$R(1)=\Pp(R)$, which is also a braided vector space called the
infinitesimal braiding. This is summarized in the following
\begin{thm} $($\text{\rm Theorem\,1.3 \cite{AC13}}$)$
Any Hopf algebra with injective antipode is a deformation of the bosonization of a Hopf algebra generated by a
cosemisimple coalgebra by a connected graded Hopf algebra in the category of Yetter-Drinfeld modules over the latter.
\end{thm}

In order to produce some new Hopf algebras by using the generalized
lifting method, one needs to consider the following questions:
\begin{itemize}
  \item {\text{\rm Question\,\uppercase\expandafter{\romannumeral1} (\cite{AC13})}} Let $C$ be a cosemisimple coalgebra and $S : C \rightarrow C$
  an injective anti-coalgebra morphism. Classify all Hopf algebras $L$ generated by $C$, belonging to the class $C$, and such that $S|_C = S$.
  \item {\text{\rm Question\,\uppercase\expandafter{\romannumeral2} (\cite{AC13})}} Given $L$ as in the previous item, classify all connected graded
  Hopf algebras $R$ in ${}_L^L\mathcal{YD}$.
  \item {\text{\rm Question\,\uppercase\expandafter{\romannumeral3} (\cite{AC13})}} Given $L$ and $R$ as in previous items, classify all liftings,
  that is, classify all Hopf algebras $H$ such that $gr\,H\cong R\sharp L$.
\end{itemize}

If $A$ is a Hopf algebra satisfying $A_{[0]}$=$L$ where $L$ is an arbitrary finite-dimensional Hopf algebra, then we call $A$ is a Hopf algebra over $L$.
Following this generalized lifting method, G.-A.~Garcia and J.-M.-J.~Giraldi \cite{GG16} determined all finite-dimensional Hopf algebras
over a Hopf algebra of dimension $8$ without the Chevalley property, and the corresponding infinitesimal braiding is an irreducible object
and obtained some new Hopf algebras of dimension $64$.

In this paper, following the work of G.-A.~Garcia and
J.-M.-J.~Giraldi \cite{GG16}, we study this question in the case
$\A_{[0]}=\C$ is the unique Hopf algebra of dimension $12$, which is
the dual of a pointed Hopf algebra and without the Chevalley
property. First, we describe the Hopf algebra structure of $\C$,
which is given by
\begin{align*}
\C:&=\langle a, b\mid a^6=1, b^2=0, ba=\xi ab\rangle,  \\
\De(a)&=a\otimes a+ \Lam^{-1}b\otimes ba^3, \quad\De(b)=b\otimes a^4+a\otimes b,
\end{align*}
where $\xi$ is a primitive root of unity and $\Lam=(\xi-1)(\xi+1)^{-1}$. Notice that, as an algebra,
$\C$ is isomorphic to a quantum linear space, but the coalgebra structure is more complicated.

In order to find some new Nichols algebra, we first describe the
Hopf algebra structure of $\C$ in detail, and determine the Drinfeld
double $\D=\D(\C^{cop})$. Then we study the representations of $\D$
and describe its simple modules, projective covers of the simple
modules and some indecomposable modules. In fact, we show in Theorem
$\ref{thmsimplemoduleD}$, there exist $6$ one-dimensional modules
$\K_{\chi^{k}}$ for $k\in Z_6$ and $30$ two-dimensional modules
$V_{i,j}$ for $(i,j)\in \Lambda$ where $\Lambda=\{(i,j)\in Z_6\times
Z_6\mid 3i\neq j\}$. Moreover, we compute the Ext-quiver and show
that $\D$ is of wild representation type.

 It is well-known that the left $\D$ module category ${}_D\M$ is equivalent to the
 category $\CYD$ of Yetter-Drinfeld modules over $\C$. Next, by using this fact, we
 translate the simple and indecomposable $\D(\C^{cop})$-modules to the Yetter-Drinfeld
 modules over $\C$. In order to study the Nichols algebras, we describe explicitly their
 structures as Yetter-Drinfeld modules and their braiding. Using the braiding, we prove
 that Nichols algebras generated by finite-dimensional non-simple indecomposable modules
 are infinite-dimensional. In particular, for any object $V$ in $\CYD$, if $\BN(V)$ is finite-dimensional,
 then $V$ must be semisimple. As it is a very difficult question to determine the structures of Nichols algebras in terms of generators and relations,
 we just give partial answer to this question and we show that Nichols algebra $\BN(V)$ is finite-dimensional
 if $V$ is isomorphic either to $\K_{\chi^{k}}$ with $k\in\{1,3,5\}$, $V_{3,1}$, $V_{3,5}$, $V_{2,2}$ or $V_{2,4}$.
 Moreover, we describe their structures in terms of generators and relations.
 $\BN(\K_{\chi^{k}})\cong \bigwedge \K_{\chi^{k}}$ for $k\in\{1,3,5\}$.
 And Nichols algebra  $\BN(V)$ over a two dimensional simple module mentioned above as an algebra is isomorphic
 to a quantum linear space while as a coalgebra is more complicated since the braiding is more complicated.
 Moreover, to our knowledge, these Nichols algebras seem to be new $6$-dimensional Nichols algebras.
 Finally, we need to study the deformations of the bosonizations of these Nichols algebras and try to find
 some $72$-dimensional Hopf algebras.

 The structure of the paper is as follows.
 In section $\ref{Preliminary}$ we first recall some basic definitions and facts about Yetter-Drinfeld module,
 Nichols algebra, Drinfeld double and representation type. In section $\ref{secDrinfelddouble}$,
 we give a detailed description of the Hopf algebra structure of $\C$ and determine the Drinfeld
 double $\D=\D(\C^{cop})$ in terms of generators and relations. In section $\ref{secPresentation}$,
 we study the representation of the Drinfeld double $\D(\C^{cop})$. We first describe the simple
 modules and the projective covers of the simple modules. Then we study its Ext-quiver, compute
 the separation diagram of the Ext-quiver and as a result show that $\D(\C^{cop})$ is of wild
 representation type. In section $\ref{secCYD}$, we determine the simple objects and projective
 covers of simple modules of $\CYD$ by using the equivalence $\DM\simeq \CYD$, and also describe
 their braiding. In section $\ref{secNicholsalg}$, we show that the Nichols algebras $\BN(\K_{\chi^{k}})$
 for $k\in\{1,3,5\}$ are finite-dimensional exterior algebras and Nichols algebra $\BN(V)$ is finite-dimensional
 if $V$ is isomorphic either to $V_{3,1}$, $V_{3,5}$, $V_{2,2}$ or $V_{2,4}$, which as an algebra is isomorphic to the quantum linear
 space but as a coalgebra is more complicated since the braiding is more complicated.
 In section $\ref{secHopfalgebra}$, we determine all finite-dimensional Hopf algebras $A$
 such that $A_{[0]}\cong \C$ and the corresponding infinitesimal braiding $V$
 is $\K_{\chi^{k}}$ with $k\in\{1,3,5\}$, $V_{3,1}$, $V_{3,5}$, $V_{2,2}$, or $V_{2,4}$
 and show that the bosonizations of the Nichols algebras associated to the simple modules
 mentioned above do not admit non-trivial deformations. As a byproduct, we get some Hopf algebras of
 dimension $72$ without the Chevalley property.

\section{Preliminaries}\label{Preliminary}
\paragraph{Conventions.} Throughout the paper, our ground field $\K$ is an algebraically closed field of characteristic zero.
Our references for Hopf algebra theory are \cite{M93} and \cite{R11} and the reference for representation theory is \cite{ARS95}.

 The notation for a Hopf algebra $H$ is standard: $\Delta$, $\epsilon$, and $S$ denote the comultiplication, the counit and the antipode.
 We use Sweedler's notation for the comultiplication and coaction, for example, for any $h\in H$, $\Delta(h)=h_{(1)}\otimes h_{(2)}$,
 $\Delta^{(n)}=(\Delta\otimes id^{\otimes n})\Delta^{(n-1)}$. Given a Hopf algebra $H$ with bijective antipode, we denote by $H^{op}$
 the Hopf algebra with the opposite multiplication, $H^{cop}$ the Hopf algebra with the opposite comultiplication, and $H^{bop}$ the Hopf algebra
 $H^{op\,cop}$. If $V$ is a $\K$-vector space, $v\in V$ and $f\in V\As$, we use either $f(v)$, $\langle f$, $v\rangle$, or $\langle v,
 f\rangle$ to denote the evaluation.
For any $n>1$, $M_n(\K)$ and $M_n\As(\K)$ denote matrix algebra and matrix coalgebra. If a simple coalgebra $C\cong M_n\As(\K)$,
we call the basis $(c_{ij})_{1\leq i,j\leq n}$ a comatrix basis if $\Delta(c_{i,j})=\sum_{k=1}^{n}c_{ik}\otimes c_{kj}$ and $\epsilon(c_{ij})=\delta_{i,j}$.

If $M$ is a left $H$-module, we denote by $Top(M)$ and $Soc(M)$ the Top (or called head) of $M$ and the sole of $M$ respectively,
by $ExtQ(M)$ the Ext-quiver or Gabriel quiver of $M$.

\subsection{Yetter-Drinfeld module and Nichols algebra}
\begin{defi}
Let $H$ be a Hopf algebra with bijective antipode. A Yetter-Drinfeld module over $H$ is a left $H$-module and a left $H$-comodule with comodule structure denoted by $\delta: V\mapsto H\otimes V, v\mapsto v_{(-1)}\otimes v_{(0)}$, such that
\begin{align*}
\delta(h\cdot v)=h_{(1)}v_{(-1)}S(h_{(3)})\otimes h_{(2)}v_{(0)},
\end{align*}
for all $v\in V,h\in H$. Let ${}^{H}_{H}\mathcal{YD}$ be the
category of Yetter-Drinfeld modules over $H$ with $H$-linear and
$H$-colinear maps as morphisms.
\end{defi}

The category ${}^{H}_{H}\mathcal{YD}$ is monoidal, braided. Indeed, if $V,W\in {}^{H}_{H}\mathcal{YD}$, $V\otimes W$ is the
tensor product over $\mathbb{C}$ with the diagonal action and coaction of $H$ and braiding
\begin{align}\label{equbraidingYDcat}
c_{V,W}:V\otimes W\mapsto W\otimes V, v\otimes w\mapsto v_{(-1)}\cdot w\otimes v_{(0)},\forall\,v\in V, w\in W.
\end{align}
Moreover, ${}^{H}_{H}\mathcal{YD}$ is rigid. That is, it has left dual and right dual, if we take $V\As$ and ${}\As V$ to the dual of
$V$ as vector space, then the left dual and the right dual are defined by
\begin{align*}
\langle h\cdot f,v\rangle=\langle f,S(h)v\rangle,\quad f_{(-1)}\langle f_{(0)},v\rangle=S^{-1}(v_{(-1)})\langle f, v_{(0)}\rangle,\\
\langle h\cdot f,v\rangle=\langle f,S^{-1}(h)v\rangle,\quad f_{(-1)}\langle f_{(0)},v\rangle=S(v_{(-1)})\langle f, v_{(0)}\rangle.
\end{align*}
We consider Hopf algebra in ${}^{H}_{H}\mathcal{YD}$. If $R$ is a Hopf algebra in ${}^{H}_{H}\mathcal{YD}$,
the space of primitive elements $P(R)=\{x\in R|\delta(x)=x\otimes 1+1\otimes x\}$ is a Yetter-Drinfeld
submodule of $R$. Moreover, for any finite-dimensional graded Hopf algebra in ${}^{H}_{H}\mathcal{YD}$, it satisfies the Poincar\'{e} duality:
\begin{pro}$\cite[Proposition\,3.2.2]{AG99}$
Let $R=\oplus_{n=0}^N R(i)$ be a graded Hopf algebra in ${}^{H}_{H}\mathcal{YD}$, and suppose that $R(N)\neq 0$. Then $\dim R(i)=\dim R(N-i)$ for any $0\leq i<N$.
\end{pro}
\begin{defi}
Let $V\in{}^{H}_{H}\mathcal{YD}$ and $I(V)\subset T(V)$ be the largest $\mathds{N}$-graded ideal and coideal $I(V)\subset T(V)$
such that $I(V)\cap V=0$. We call $\mathcal{B}(V)=T(V)/I(V)$ the Nichols algebra of $V$. Then $\mathcal{B}(V)=\oplus_{n\geq 0}\mathcal(B)^n(V)$
is an $\mathds{N}$-graded Hopf algebra in ${}^{H}_{H}\mathcal{YD}$.
\end{defi}
\begin{lem}$\cite{AS02}$
The Nichols algebra of an object $V\in{}^{H}_{H}\mathcal{YD}$ is the
(up to isomorphism) unique $\mathds{N}$-graded Hopf algebra $R$ in
${}^{H}_{H}\mathcal{YD}$ satisfying the following properties:
\begin{align*}
&R(0)=\mathds{k}, \quad R(1)=V,\\
&R(1)\ \text{\ generates the algebra R},\\
&P(R)=V.
\end{align*}
\end{lem}

Nichols algebras play a key role in the classification of pointed Hopf algebras, and
we close this subsection by giving the explicit relation between $V$ and $V\As$ in ${}^{H}_{H}\mathcal{YD}$.
\begin{pro}$\cite[Proposition\,3.2.30]{AG99}$\label{proNicholsdual}
Let $V$ be an object in ${}_H^H\mathcal{YD}$. If $\BN(V)$ is finite-dimensional, then $\BN(V\As)\cong \BN(V)\As$.
\end{pro}

\subsection{Radford biproduct construction}
Let $R$ be a bialgebra (rep.~Hopf algebra) in ${}^{H}_{H}\mathcal{YD}$ and denote the coproduct
by $\Delta_R(r)=r^{(1)}\otimes r^{(2)}$. We define the Radford biproduct $R\#H$. As a vector space,
$R\#H=R\otimes H$ and the multiplication and comultiplication are given by the smash product and smash-coproduct, respectively:
\begin{align}
(r\#g)(s\#h)&=r(g_{(1)}\cdot s)\#g_{(2)}h,\\
\Delta(r\#g)&=r^{(1)}\#(r^{(2)})_{(-1)}g_{(1)}\otimes (r^{(2)})_{(0)}\#g_{(2)}.
\end{align}
Clearly, the map $\iota:H\rightarrow R\#H, h\mapsto 1\# h,\ \forall h\in H$, and the map
$\pi:R\#H\rightarrow H,r\#h\mapsto \epsilon_R(r)h,\ \forall r\in R, h\in H$ such that $\pi\circ\iota=id_H$. Moreover, $R=(R\#H)^{coH}$.

Let $R, S$ be bialgebras (resp. Hopf algebra) in ${}^{H}_{H}\mathcal{YD}$ and $f:R\rightarrow S$ be a bialgebra
morphisms in ${}^{H}_{H}\mathcal{YD}$. $f\#id:R\#H\rightarrow S\#H$ defined by $(f\#id)(r\#h)=f(r)\#h,
\forall r\in R, h\in H$. In fact, $R\rightarrow R\#H$ and $f\mapsto f\# id$ describes a
functor from the category of bialgebras (resp. Hopf algebras) in ${}^{H}_{H}\mathcal{YD}$ and their morphisms
to the category of usual bialgebras (resp. Hopf algebras).

Conversely, if $A$ is a bialgebra (resp. Hopf algebra) and $\pi:A\rightarrow H$ a bialgebra admitting a bialgebra
section $\iota:H\rightarrow A$ such that $\pi\circ\iota=id_H$ (we call $(A,H)$ a Radford pair for convenience),
$R=A^{coH}=\{a\in A\mid(id\otimes\pi)\Delta(a)=a\otimes 1\}$ is a bialgebra (resp. Hopf algebra) in ${}^{H}_{H}\mathcal{YD}$
and $A\simeq R\#H$, whose Yetter-Drinfeld module and coalgebra structures are given by:
\begin{align*}
h\cdot r&=h_{(1)}rS_A(h_{(2)}),
\delta(r)=(\pi\otimes id)\Delta_A(r),\\
\Delta_R(r)&=r_{(1)}(\iota S_H(\pi(r_{(2)})))\otimes r_{(3)},
\epsilon_R=\epsilon_A|_R,\\
S_R(r)&=(\iota\pi(r_{(1)}))S_A(r_{(2)}),\quad\text{if $A$ is a Hopf algebra}.
\end{align*}

\subsection{Drinfeld double}
\begin{defi}
Let $H$ be a finite-dimensional Hopf algebra with bijective antipode $S$ over $\K$. The Drinfeld double $\D(H)=H^{\ast\,cop}\otimes H$
is a Hopf algebra with the tensor product coalgebra structure and algebra structure defined by
\begin{align}\label{equDrinfelddouble}
(p\otimes a)(q\otimes b)=p\langle q_{(3)}, a_{(1)}\rangle q_{(2)}\otimes a_{(2)}\langle q_{(1)}, S^{-1}(a_{(3)})\rangle.
\end{align}
\end{defi}

By $\cite[Proposition 10.6.16]{M93}$, the category
${}_{\D(H)}{\M}$ of left modules is equivalent to the category ${}_H\mathcal{YD}^H$ of Yetter-Drinfeld
modules. But ${}_H\mathcal{YD}^H$ is equivalent
to the category ${}_{H^{cop}}^{H^{cop}}\mathcal{YD}$ of Yetter-Drinfeld modules.
Thus we have the following
result.
\begin{pro}$\cite{M93}$
Let $H$ be a finite-dimensional Hopf algebra with bijective antipode $S$ over $\K$.
Then the category ${}_{\D(H^{cop})}\M$ of left modules is equivalent to the category ${}_H^H\mathcal{YD}$
of Yetter-Drinfeld modules.
\end{pro}

\subsection{Ext-Quiver and representation type}
Let $\Lambda$ be a finite-dimensional algebra, and $\{S_1, ...,
S_n\}$ a complete list of non-isomorphic simple modules. The
Ext-Quiver of $\Lambda$ is the quiver $ExtQ(\Lambda)$ with vertices
$1,2,...,n$ and $\dim Ext_{\Lambda}^1(S_i,\,S_j)$ arrows from the
vertex $i$ to $j$. The separated quiver of $\Lambda$ is constructed
as follows: The set of vertices is $\{S_1, ..., S_n, S_1\As, ...,
S_n\As\}$ and we write $\dim\,Ext_{\Lambda}^1(S_i, S_j)$ arrows from
$S_i$ to $S_j\As$.

Let us denote by $\Gamma_{\Lambda}$ the underlying graph of the separated quiver of $\Lambda$. For a finite dimensional algebra with radical square zero,
the following Proposition is well-known.
\begin{pro}\label{proRepType}
Let $\Lambda$ be a finite-dimensional algebra with radical square zero. Then $\Lambda$ is of finite (resp. tame)
representation type if and only if $\Gamma_{\Lambda}$ is a finite (resp. affine) disjoint union of Dynkin diagrams.
\end{pro}

\section{The Hopf algebra $\C$ and Drinfeld double $D(\C)$}\label{secDrinfelddouble}
Throughout the paper, we fix $\xi$ a primitive $6$-th root of unity. For the classification of Hopf algebra of dimension $12$,
the semisimple case was classified by N.~Fukuda \cite{F97}, the pointed nonsemisimple case was given by Andruskiewitsch and
Natale \cite{AN01}, and the full classification was done by Natale \cite{Na02}. The main result is the following
\begin{thm} $\cite[Theorem\,0.1]{Na02}$
Let $H$ be a nonsemisimple Hopf algebra of dimension $12$ over $\K$. Then either $H$ or $H\As$ is pointed.
\end{thm}
The following list is (up to isomorphism) all the pointed Hopf algebras of dimension $12$ over $\K$.
\begin{align*}
\A_0:&=\langle g,x\mid g^6=1, x^2=0, gx=-xg\;\rangle;  \De(g)=g\otimes g, \De(x)=x\otimes 1+g\otimes x.\\
\A_1:&=\langle g,x\mid g^6=1, x^2=1{-}g^2, gx=-xg\;\rangle;  \De(g)=g\otimes g, \De(x)=x{\otimes} 1+g{\otimes} x.\\
\BN_0:&=\langle g,x\mid g^6=1, x^2=0, gx=-xg\;\rangle;  \De(g)=g\otimes g, \De(x)=x\otimes 1+g^3\otimes x.\\
\BN_1:&=\langle g,x\mid g^6=1, x^2=0, gx=\xi xg\;\rangle;  \De(g)=g\otimes g, \De(x)=x\otimes 1+g^3\otimes x.
\end{align*}
We also have that $\A_0\As\cong \BN_1$ and $\BN_0$ is self-dual. In particular, the Hopf algebra $\C:=\A_1^{\ast}$ is
the unique Hopf algebra of dimension $12$, which is neither pointed nor semisimple nor has the Chevalley property.
Moreover, $\C$ is generated as an algebra by its simple subcoalgebra, and its coradical $\C_0\simeq \K\oplus \K \oplus C\oplus D$, where $C$ and $D$ are simple coalgebras of dimension $4$.

In order to give the Hopf algebra structure of $\C$ explicitly, we first describe the simple representations of $\A_1$.
\begin{lem}\label{lem1}
The simple one-dimensional representations of $\A_1$ are $\epsilon$ and $\chi:\A_1\mapsto \K$, where
\begin{align}
\chi(g)=-1,\quad \chi(x)=0.
\end{align}
The simple two-dimensional representations of $\A_1$ on a given basis are $\rho_1, \rho_2:\A_1\mapsto \M(2,\K)$, where
\begin{align}
    \rho_1(g)&=\left(\begin{array}{ccc}
                                   \xi & 0\\
                                   0 & -\xi
                                 \end{array}\right),\quad
    \rho_1(x)=\left(\begin{array}{ccc}
                                   0 & 1-\xi\\
                                   1+\xi & 0
                                 \end{array}\right);\\
    \rho_2(g)&=\left(\begin{array}{ccc}
                                   -\xi^{-1} & 0\\
                                   0 & \xi^{-1}
                                 \end{array}\right),\quad
    \rho_2(x)=\left(\begin{array}{ccc}
                                   0 & 1-\xi^{-1}\\
                                   1+\xi^{-1} & 0
                                 \end{array}\right).
\end{align}
\end{lem}
\begin{proof}
Let $\beta\in G(\A_1\As)=\hom(\A_1,\K)$. Since $g^6=1$ and $gx=-xg$, $\beta(g)$ is a $6$-th root of unity and thus $\beta(x)=0$. From the relation $x^2=1-g^2$, we have that $\beta(g)^2=1$, i.e., $\beta(g)=\pm 1$. Thus the simple modules of $\A_1$ are given by $\epsilon$ and $\chi$.

For the two-dimensional simple modules, since $g^6=1$, we can choose a basis such that the matrices defining the module action are of the form
\begin{align*}
    \rho(g)&=\left(\begin{array}{ccc}
                                   g_1 & 0\\
                                   0 &   g_2
                                 \end{array}\right),\quad
    \rho(x)=\left(\begin{array}{ccc}
                                   x_1 & x_2\\
                                   x_3 & x_4
                                 \end{array}\right),
\end{align*}
where $g_1^6=g_2^6=1$. From the relation $gx=-xg$, it follows that $x_1=0=x_3$ and $(g_1+g_2)x_3=0=(g_1+g_2)x_2$. If $g_1+g_2\neq 0$, then $x_3=0=x_2$, which implies that $\rho(x)=0$, a contradiction. Thus we have $g_1=-g_2$. Moreover, $x_2\neq 0$ and $x_3\neq 0$ since the representation is simple. By the relation $x^2=1-g^2$, we have that $x_2x_3=1-g_1^2$. Since $x_2x_3\neq 0$, $g_1^2\neq 1$, which implies $g_1\neq \xi^3$. If $g_1=\xi$ or $\xi^4$, then the module is isomorphic to $\rho_1$. If $g_1=\xi^5$ or $\xi^2$, then the module is isomorphic to $\rho_2$.
\end{proof}
Let $(\K^2, \rho_i)_{i=1,2}$ be the $2$-dimensional representations given in Lemma $\ref{lem1}$. Let ${(E_{ij})}_{i,j=1,2}$ be the coordinate functions of $\M(2,\K)$. And let $C_{ij}:=E_{ij}\circ \rho_1, D_{ij}:=E_{ij}\circ \rho_2$, we can regard $\E_C:=\{C_{ij}\}_{i,j=1,2}$ and $\E_D:=\{D_{ij}\}_{i,j=1,2}$ as comatrix basis of the simple subcoalgebras of $\C$ isomorphic to $C$ and $D$, respectively.
\begin{lem}
The elements of $\E_C$ and $\E_D$ satisfy:
\begin{align*}
S(C_{12})&=D_{12},\quad S(C_{21})=D_{21},\quad S(C_{11})=D_{22},\quad S(C_{22})=D_{11},\\
S(D_{12})&=-C_{12},\quad S(D_{21})=-C_{21},\quad S(D_{11})=C_{22},\quad S(D_{22})=C_{11};\\
C_{11}^3&=\chi,\quad C_{22}^3=\epsilon, \quad C_{11}C_{22}=C_{22}C_{11}, \quad C_{11}^2C_{22}=\epsilon,\\
C_{12}^2&=0=C_{21}^2,\quad C_{12}C_{21}=0=C_{21}C_{12},\\
C_{11}C_{12}&=\xi C_{12}C_{11},\quad C_{11}C_{21}=\xi C_{21}C_{11},\\
C_{11}C_{12}&=\Lam C_{22}C_{21},\quad C_{11}C_{21}=\Lam^{-1}C_{22}C_{12};\\
\De(C_{11})&=C_{11}\otimes C_{11}+C_{12}\otimes C_{21},\quad \De(C_{12})=C_{11}\otimes C_{12}+C_{12}\otimes C_{22},\\
\De(C_{21})&=C_{21}\otimes C_{11}+C_{22}\otimes C_{21},\quad \De(C_{22})=C_{21}\otimes C_{12}+C_{22}\otimes C_{22};\\
\epsilon(C_{11})&=\epsilon(C_{22})=1,  \quad \epsilon(C_{12})=\epsilon(C_{21})=0;\\
S(C_{11})&=C_{11}^5,\quad S(C_{12})=\Lam\xi C_{22}C_{21},\\
S(C_{21})&=\Lam^{-1}\xi^{-2}C_{22}C_{12},\quad
S(C_{22})=C_{11}^2,
\end{align*}
where $\Lam=(\xi-1)(\xi+1)^{-1}$.
\end{lem}
\begin{proof}
Note that $\{g^i,\,g^ix\}_{i=0}^5$ is a basis of $\A_1$, $S(g)=g^{-1}$ and $S(x)=-g^{-1}x=xg^{-1}$. Since the antipode of $\C$ is defined by $S(c)=c\circ S_{\A_1}$, we have
\begin{align*}
\langle S(C_{11}), g^i\rangle&=\langle C_{11}, S(g^i)\rangle=\langle E_{11}\rho_1, S(g^i)\rangle =\xi^{-i}=\langle D_{22}, g^i\rangle,\\
\langle S(C_{11}), g^ix\rangle&=\langle C_{11}, S(g^ix)\rangle=\langle E_{11}\rho_1, S(g^ix)\rangle  =0=\langle D_{22}, g^ix\rangle,\\
\langle S(C_{22}), g^i\rangle&=\langle C_{22}, S(g^i)\rangle=\langle E_{11}\rho_1, S(g^i)\rangle  =(-\xi^{-1})^{i}=\langle D_{11}, g^i\rangle,\\
\langle S(C_{22}), g^ix\rangle&=\langle C_{22}, S(g^ix)\rangle=\langle E_{11}\rho_1, S(g^ix)\rangle =0=\langle D_{11}, g^ix\rangle.
\end{align*}
Then $S(C_{11})=D_{22}$, $S(C_{22})=D_{11}$.  Similarly, we have  $S(C_{12})=D_{12}, S(C_{21})=D_{21}, S(C_{11})=D_{22}, S(C_{22})=D_{11}, S(D_{12})=-C_{12}, S(D_{21})=-C_{21}$.

Since the multiplication of $\C$ is defined by the convolution of $\hom(\A_1,\K)$, and note that
$\Delta(g^i)=g^i\otimes g^i$, $\Delta(g^ix)=g^ix\otimes g^i+g^{i+1}\otimes g^ix$. Clearly, we have
\begin{align*}
\langle C_{11}^3, g^i\rangle&=\langle C_{11}, g^i\rangle \langle C_{11}, g^i\rangle\langle C_{11}, g^i\rangle=(-1)^i=\chi(g^i),\\
\langle C_{11}^3, g^ix\rangle&=\langle C_{11}, g^ix\rangle \langle C_{11}, g^i\rangle\langle C_{11}, g^i\rangle+\langle C_{11}, g^{i+1}\rangle \langle C_{11}, g^ix\rangle\langle C_{11}, g^i\rangle\\&\quad+\langle C_{11}, g^{i+1}\rangle \langle C_{11}, g^{i+1}\rangle\langle C_{11}, g^ix\rangle=0=\chi(g^i).
\end{align*}
Then $C_{11}^3=\chi$. Similarly, we can prove $C_{22}^3=\epsilon$, $C_{11}C_{22}=C_{22}C_{11}$, $C_{11}^3C_{22}=C_{22}C_{11}^3=1$.
After a similar computation as above, it follows that $C_{12}^2=0=C_{21}^2, C_{12}C_{21}=0=C_{21}C_{12}$.
It is clear that $\langle C_{11}C_{12}, g^i\rangle=0=\langle C_{12}C_{11}, g^i\rangle=0$. And
\begin{align*}
\langle C_{11}C_{12}, g^ix\rangle
&=\langle C_{11}, g^ix\rangle \langle C_{12}, g^i\rangle+\langle C_{11}, g^{i+1}\rangle \langle C_{12}, g^ix\rangle\\
&=\langle C_{11}, g^{i+1}\rangle \langle C_{12}, g^ix\rangle=\xi^{i+1}\langle C_{12}, g^ix\rangle,\\
\langle C_{12}C_{11}, g^ix\rangle
&=\langle C_{12}, g^ix\rangle \langle C_{11}, g^i\rangle+\langle C_{12}, g^{i+1}\rangle \langle C_{11}, g^ix\rangle\\
&=\langle C_{11}, g^{i}\rangle \langle C_{12}, g^ix\rangle=\xi^i\langle C_{12}, g^ix\rangle.
\end{align*}
Thus we have $C_{11}C_{12}=\xi C_{12}C_{11}$. Similarly, we have $C_{11}C_{21}=\xi C_{12}C_{21}$. Now we can claim that $C_{11}C_{12}=\Lam C_{22}C_{21}$.
Indeed, $\langle C_{22}C_{21}, g^i\rangle=0$ and
\begin{align*}
\langle C_{22}C_{21}, g^ix\rangle
&=\langle C_{22}, g^ix\rangle \langle C_{21}, g^i\rangle+\langle C_{22}, g^{i+1}\rangle \langle C_{21}, g^ix\rangle\\
&=\langle C_{22}, g^{i+1}\rangle \langle C_{21}, g^ix\rangle
=(-\xi)^{i+1}\langle C_{21}, g^ix\rangle\\
&=(1+\xi)(\xi-1)^{-1}\xi^{i+1}\langle C_{12}, g^ix\rangle\\
&=(1+\xi)(\xi-1)^{-1}\langle C_{11}C_{12}, g^ix\rangle.
\end{align*}
Similarly, $C_{22}C_{12}=\Lam C_{11}C_{21}$. And the rest claims are clear.
\end{proof}

Since $\C$ can be generated by its simple subcoalgebra, and from the above lemma, let $a=S(C_{11})$, $b=\Lam S(C_{21})$, we have the following
\begin{pro}\label{proStrucOfC}
\begin{enumerate}
  \item $\C$ is generated as an algebra by the elements $a$ and $b$ satisying
  $$
   a^6=1,\quad b^2=0,\quad ba=\xi ab.
  $$
  \item $A$ basis of $\C$ as vector space is given by
  \begin{align*}
  \{1, a, a^2, a^3, a^4, a^5, b, ba, ba^2, ba^3, ba^4, ba^5\}.
  \end{align*}
  \item The coalgebra structure is given by
  \begin{align*}
  \De(1)&=1\otimes 1,\quad
  \De(a)=a\otimes a+ \Lam^{-1}b\otimes ba^3,\\
  \De(a^2)&=a^2\otimes a^2+\xi^{-1}\Lam^{-1}ba\otimes ba^4,\quad
  \De(a^3)=a^3\otimes a^3,\\
  \De(a^4)&=a^4\otimes a^4+\Lam^{-1}ba^3\otimes b,\quad
  \De(a^5)=a^5\otimes a^5+\xi^{-1}\Lam^{-1}ba^4\otimes ba,\\
  \De(b)&=b\otimes a^4+a\otimes b,\quad
  \De(ba)=ba\otimes a^5+a^2\otimes ba,\\
  \De(ba^2)&=a^3\otimes ba^2+ba^2\otimes 1,\quad
  \De(ba^3)=ba^3\otimes a+a^4\otimes ba^3,\\
  \De(ba^4)&=a^5\otimes ba^4+ba^4\otimes a^2,\quad
  \De(ba^5)=ba^5\otimes a^3+1\otimes ba^5,\\
  \epsilon(a)&=1,\quad\epsilon(b)=0.
  \end{align*}
  \item The antipode is given by
  \begin{align*}
  S(a)=a^5, \quad S(b)=\xi^{-2}ba.
  \end{align*}
\end{enumerate}
\end{pro}

\begin{rmk}
$\C$ as an algebra, is isomorphic to $\BN_1$, but the coalgebra structure is more complicated. Indeed, by \cite[Theorem A.1]{Ma08}, $\A_1$ is a Hopf 2-cocycle deformation of $\A_0$, thus as the dual of $\A_1$, $\C$ is naturally the Drinfeld twist deformation of $\BN_1$ since $\A_0\As\cong \BN_1$ as Hopf algebras.
\end{rmk}
\begin{rmk}
Note that $a^3$ is a group-like element and $ba^2$, $ba^5$ are skew-primitive. In particular, the sub-algebra generated by $a^3$ and $ba^2$ is a Hopf subalgebra which is isomorphic to the $4$-dimensional Sweedler Hopf algebra.
\end{rmk}
\begin{rmk}\begin{enumerate}
             \item Denote the basis of $\C^{\ast}$ dual to the basis of $\C$  by
\begin{align*}
\{1\As,\Aa,\Ab,\Ac,\Ad,\Ae,\B,\BAa,\BAb,\BAc,\BAd,\BAe\}.
\end{align*}
From the multiplication table induced by the relations of $\C$, we have
\begin{align*}
\De(1\As)&=1\As\otimes 1\As+\Aa\otimes \Ae+\Ae\otimes\Aa+\Ab\otimes\Ac\\
         &\quad+\Ac\otimes\Ab+\Ac\otimes\Ac,\\
\De(\Aa)&=\Aa\otimes 1\As+1\As\otimes\Aa+\Ab\otimes\Ae+\Ae\otimes\Ab\\
        &\quad+\Ac\otimes\Ad+\Ad\otimes\Ac,\\
\De(\Ab)&=1\As\otimes\Ab+\Ab\otimes 1\As+\Aa\otimes\Aa+\Ac\otimes\Ae\\
        &\quad+\Ae\otimes\Ac+\Ad\otimes\Ad,\\
\De(\Ac)&=1\As\otimes\Ac+\Ac\otimes 1\As+\Aa\otimes\Ab+\Ab\otimes\Aa\\
        &\quad+\Ad\otimes\Ae+\Ae\otimes\Ad,\\
\De(\Ad)&=1\As\otimes\Ad+\Ad\otimes 1\As+\Aa\otimes\Ac+\Ac\otimes\Aa\\
        &\quad+\Ab\otimes\Ab+\Ae\otimes\Ae,\\
\De(\Ae)&=1\As\otimes\Ae+\Ae\otimes 1\As+\Aa\otimes\Ad+\Ad\otimes\Ae\\
        &\quad+\Ab\otimes\Ac+\Ac\otimes\Ab,\\
\De(\B)&=1\As\otimes\B+\B\otimes 1\As+\xi^{-1}\Aa\otimes\BAe+\BAe\otimes\Aa\\
       &\quad+\xi^{-2}\Ab\otimes\BAd+\BAd\otimes\Ab+\xi^{-3}\Ac\otimes\BAc\\
       &\quad+\BAc\otimes\Ac+\xi^{-4}\Ad\otimes\BAb+\BAb\otimes\Ad\\
       &\quad+\xi^{-5}\Ae\otimes\BAa+\BAa\otimes\Ae,\\
\De(\BAa)&=1\As\otimes\BAa+\BAa\otimes 1\As+\xi^{-1}\Aa\otimes\B+\B\otimes\Aa\\
       &\quad+\xi^{-2}\Ab\otimes\BAe+\BAe\otimes\Ab+\xi^{-3}\Ac\otimes\BAd\\
       &\quad+\BAd\otimes\Ac+\xi^{-4}\Ad\otimes\BAc+\BAc\otimes\Ad\\
       &\quad+\xi^{-5}\Ae\otimes\BAb+\BAb\otimes\Ae,
\end{align*}
\begin{align*}
\De(\BAb)&=1\As\otimes\BAb+\BAb\otimes 1\As+\xi^{-1}\Aa\otimes\BAa+\BAa\otimes\Aa\\
       &\quad+\xi^{-2}\Ab\otimes\B+\B\otimes\Ab+\xi^{-3}\Ac\otimes\BAe\\
       &\quad+\BAe\otimes\Ac+\xi^{-4}\Ad\otimes\BAd+\BAd\otimes\Ad\\
       &\quad+\xi^{-5}\Ae\otimes\BAc+\BAc\otimes\Ae,     \\
\De(\BAc)&=1\As\otimes\BAc+\BAc\otimes 1\As+\xi^{-1}\Aa\otimes\BAb+\BAb\otimes\Aa\\
       &\quad+\xi^{-2}\Ab\otimes\BAe+\BAe\otimes\Ab+\xi^{-3}\Ac\otimes\B\\
       &\quad+\B\otimes\Ac+\xi^{-4}\Ad\otimes\BAe+\BAe\otimes\Ad\\
       &\quad+\xi^{-5}\Ae\otimes\BAd+\BAd\otimes\Ae,\\
\De(\BAd)&=1\As\otimes\BAd+\BAd\otimes 1\As+\xi^{-1}\Aa\otimes\BAc+\BAc\otimes\Aa\\
       &\quad+\xi^{-2}\Ab\otimes\BAb+\BAb\otimes\Ab+\xi^{-3}\Ac\otimes\BAa\\
       &\quad+\BAa\otimes\Ac+\xi^{-4}\Ad\otimes\B+\B\otimes\Ad\\
       &\quad+\xi^{-5}\Ae\otimes\BAe+\BAe\otimes\Ae,\\
\De(\BAe)&=1\As\otimes\BAe+\BAe\otimes 1\As+\xi^{-1}\Aa\otimes\BAd+\BAd\otimes\Aa\\
       &\quad+\xi^{-2}\Ab\otimes\BAc+\BAc\otimes\Ab+\xi^{-3}\Ac\otimes\BAb\\
       &\quad+\BAb\otimes\Ac+\xi^{-4}\Ad\otimes\BAa+\BAa\otimes\Ad\\
       &\quad+\xi^{-5}\Ae\otimes\B+\B\otimes\Ae.
\end{align*}
In particular, we have
\begin{align*}
\Delta(\widetilde{x})=\widetilde{x}\otimes \epsilon+\widetilde{g}\otimes \widetilde{x},\quad
\Delta(\widetilde{g})=\widetilde{g}\otimes \widetilde{g},
\end{align*}
where $\widetilde{x}=\B+\BAa+\BAb+\BAc+\BAd+\BAe$, $\widetilde{g}=1\As+\xi^{-1}\Aa+\xi^{-2}\Ab+\xi^{-3}\Ac+\xi^{-4}\Ad+\xi^{-5}\Ae$.
\item Let $\alpha\in G(\C\As)=Alg(\C,\K)$. Since $a^6=1,b^2=0,ba=\xi ab$, we have $\alpha(a)$ is a $6$-th root of unity and $\alpha(b)=0$. Thus
    \begin{align*}
    G(\C\As)=\{\alpha_i=1\As{+}\xi^{-i}\Aa{+}\xi^{-2i}\Ab{+}\xi^{-3i}\Ac{+}\xi^{-4i}\Ad{+}\xi^{-5i}\Ae\}.
  \end{align*}
   Note that $\alpha_0=\epsilon$, $\alpha_i=(\alpha_1)^i$, and $G(\C\As)\simeq Z_6$ with generator $\alpha_1$ or $\alpha_5$.
\end{enumerate}

\end{rmk}
In order to compute the structure of the Drinfeld double $D(\C^{cop})$ of $\C^{cop}$ in terms of generators and relations, we have the following lemma which builds the isomorphism $\A_1\cong \C\As$ explicitly.
\begin{lem}\label{lemAtoCdual}
The algebra map $\phi:\A_1\rightarrow \C\As$ given by
\begin{align*}
\phi(g)&=\alpha_1=1\As+\xi^{-1}\Aa+\xi^{-2}\Ab+\xi^{-3}\Ac+\xi^{-4}\Ad+\xi^{-5}\Ae,\\
\phi(x)&=\theta(\B+\BAa+\BAb+\BAc+\BAd+\BAe), \quad \theta^2=\xi-1,
\end{align*}
is a Hopf algebra map.
\end{lem}
\begin{proof}
It is clear that $\phi$ is a Hopf algebra map after a direct calculation. Hence, the image of $\phi$ is a Hopf subalgebra of $\C\As$ of dimension bigger than $6$, since it contains the group algebra $\K G(\C\As)$ and the image of $x$. Thus, by the Nichols-Zoller theorem, $\phi$ is surjective and whence an isomorphism since their dimensions are the same.
\end{proof}

\begin{rmk}\label{rmkAtoCdual}
Let $\{ g^i, g^ix\}_{0\leq i,j<6}$ be a linear basis of $\A_1$. We have
\begin{align*}
\phi(g^i)&=\alpha_i=1\As+\xi^{-i}\Aa+\xi^{-2i}\Ab+\xi^{-3i}\Ac+\xi^{-4i}\Ad+\xi^{-5i}\Ae,\\
\phi(g^ix)&=\theta(\xi^{-i}\B+\xi^{-2i}\BAa+\xi^{-3i}\BAb+\xi^{-4i}\BAc+\xi^{-5i}\BAd+\BAe).
\end{align*}
\end{rmk}

Now we try to describe the Drinfeld double $\D:=\D(\C^{cop})$ of $\C^{cop}$.
\begin{pro}
$\D:=\D(\C^{cop})$ as a coalgebra is isomorphic to the tensor coalgebra $\A_1^{bop}\otimes \C^{cop}$, and as an algebra is generated by the elements $a, b, g, x$ satisfying the relations in $\C^{cop}$, the relations in $\A_1^{bop}$ and
\begin{align*}
ag&=ga,\quad ax+\xi^{-2}xa=\Lam^{-1}\theta\xi^{-2}(ba^3-gb),\\
bg&=-gb,\quad bx+\xi^{-2}xb=\theta\xi^{-2}(a^4-ga).
\end{align*}
\end{pro}
\begin{proof}
Note that
\begin{align*}
\De_{\A_2}^{2}(g)&=g\otimes g\otimes g,\quad \De_{\A_2}^{2}(x)=x\otimes 1\otimes 1+g\otimes x\otimes 1+g\otimes g\otimes x,\\
\De_{\C}^2(a)&=a\otimes a\otimes a+\Lam^{-1}b\otimes ba^3\otimes a+\Lam^{-1}b\otimes a^4\otimes ba^3+\Lam^{-1}a\otimes b\otimes ba^3,\\
\De_{\C}^2(b)&=b\otimes a^4\otimes a^4+a\otimes b\otimes a^4+a\otimes a\otimes b+\Lam^{-1}b\otimes ba^3\otimes b.
\end{align*}
By equation $\eqref{equDrinfelddouble}$, we have
\begin{align*}
ag&=\langle g,a\rangle ga\langle g,S(a)\rangle=ga,\\
bg&=\langle g,a^4\rangle gb\langle g,S(a)\rangle =-gb,\\
ax&=\Lam^{-1}\langle 1,a\rangle ba^3\langle x,S(b)\rangle+\langle 1,a\rangle xa\langle g,S(a)\rangle+\Lam^{-1}\langle x,ba^3\rangle gb\langle g,S(a)\rangle\\
  &=\Lam^{-1}\theta\xi^{-2}ba^3+\xi^{-5}xa+\Lam^{-1}\theta\xi^{-5}gb,\\
bx&=\langle 1,a^4\rangle a^4\langle x,S(b)\rangle+\langle 1,a^4\rangle xb\langle g,S(a)\rangle+\langle x,b\rangle ga\langle g,S(a)\rangle\\
  &=\theta\xi^{-2}a^4+\xi^{-5}xb+\theta\xi^{-5}ga.
\end{align*}
\end{proof}

\section{Presentation of the Drinfeld double $\D(\C^{cop})$}\label{secPresentation}
In this section, we study the representations of the Drinfeld double $\D(\C^{cop})$. We first describe the simple modules and the projective covers of the simple modules. Then we study its Ext-quiver, compute the separation diagram and as a result show that $\D(\C^{cop})$ is of wild representation type.
Now we begin by describing the one-dimensional $\D$-modules.
\begin{lem}\label{onesimple}
There are six non-isomorphic one-dimensional simple modules given by the characters $\chi^i,\,0\leq i<6$, where
\begin{align*}
\chi^i(a)=\xi^i, \quad \chi^i(b)=0,\quad \chi^i(g)=(-1)^i,\quad \chi^i(x)=0.
\end{align*}
Moreover, any one-dimensional $\D$-module is isomorphic to $\K_{\chi^i}$ for some $0\leq i<6$.
\end{lem}
\begin{proof}
Let $\chi\in G(\D\As)=\hom(\D,\K)$. Since $a^6=1=g^6$, we have $\chi(a)$ and $\chi(g)$ are both $6$-th roots of unity.
From $b^2=0$, $gb=-bg$ and $gx=-xg$, we have that $\chi(b)=\chi(x)=0$, and whence $\chi(g)^2=1$ since $x^2=1-g^2$. From
the relation $bx+\xi^{-2}xb=\theta\xi^{-2}(a^4-ga)$, we have $\chi(a)^3=\chi(g)$. Thus $\chi$ is completely determined by $\chi(a)$. Let $\chi(a)=\xi^i$ for some $i\in Z_6$, for different $i$, it is clear that these modules are pairwise non-isomorphic and any one-dimensional $\D$-module is isomorphic to $\K_{\chi^i}$ for some $0\leq i<6$.
\end{proof}
Next, we describe two-dimensional simple $\D$-modules. To this end, let us consider the finite subset of $Z_6\times Z_6$ given by
\begin{align*}
\Lambda=\{(i,j)\in Z_6\times Z_6\mid 3i\neq j\}.
\end{align*}
Clearly, $|\Lambda|=30$.
\begin{lem}\label{twosimple}
For any pair $(i,j)\in\Lambda$, there exists a simple left $\D$-module $V_{i,j}$ of dimension $2$. If we denote $\Lam_1=\xi^i$ and $\Lam_2=\xi^j$, the action on a fixed basis is given by
\begin{align*}
    [a]_{i,j}&=\left(\begin{array}{ccc}
                                   \Lam_1 & 0\\
                                   0 &    \xi\Lam_1
                                 \end{array}\right),\quad
    [b]_{i,j}=\left(\begin{array}{ccc}
                                   0 & 1\\
                                   0 & 0
                                 \end{array}\right),\quad
    [g]_{i,j}=\left(\begin{array}{ccc}
                                   \Lam_2 & 0\\
                                   0 & -\Lam_2
                                 \end{array}\right),\\
    [x]_{i,j}&=\left(\begin{array}{ccc}
                                   0 & \theta^{-1}\xi^{2}\Lam_1^{-1}(\Lam_1^3+\Lam_2)\\
                                   \theta\xi^{-2}\Lam_1(\Lam_1^3-\Lam_2) & 0
                                 \end{array}\right).
\end{align*}
\end{lem}
\begin{proof}
Since $a^6=g^6=1$ and $ga=ag$, we can choose a basis of the two dimensional simple $\D$-module $V$ of such that the matrices defining the action on $V$ are of the form
\begin{align*}
    [a]&=\left(\begin{array}{ccc}
                                   a_1 & 0\\
                                    0  & a_2
                                 \end{array}\right),\quad
    [b]=\left(\begin{array}{ccc}
                                   b_1 & b_2\\
                                   b_3 & b_4
                                 \end{array}\right),\\
    [g]&=\left(\begin{array}{ccc}
                                   g_1 & 0\\
                                   0   & g_2
                                 \end{array}\right),\quad
    [x]=\left(\begin{array}{ccc}
                                   x_1 & x_2\\
                                   x_3 & x_4
                                 \end{array}\right),
\end{align*}
where $a_1^6=1=a_2^6$ and $g_1^6=1=g_2^6$. From the relation $gb=-bg$, we have that
\begin{align*}
        \left(\begin{array}{ccc}
                                   g_1b_1 & g_1b_2\\
                                   g_2b_3 & g_2b_4
                                 \end{array}\right)=-
        \left(\begin{array}{ccc}
                                   b_1g_1 & b_2g_2\\
                                   b_3g_1  & b_4g_2
                                 \end{array}\right).
\end{align*}
Thus we get $2g_1b_1=0=2g_2b_4$ and $(g_1+g_2)b_3=0=(g_1+g_2)b_2$, which implies that $b_1=0=b_4$ since $g_1, g_2\neq 0$. Similarly, from the relation $gx=-xg$, we have $x_1=0=x_4$ and $(g_1+g_2)x_3=0=(g_1+g_2)x_2$. If $g_1+g_2\neq 0$, then $b_3=0=b_2,\,x_3=0=x_2$, that is, $[x], [b]$ are zero matrices, which implies $V$ can be decomposed as a $\D$-module, a contradiction. Thus we have $g_1=-g_2$.

From the relation $b^2=0$, we have that $b_2b_3=0$ and assume $b_3=0$. If $b_2=0$, it is clear that $V$ is simple $\Leftrightarrow x_2x_3\neq 0$. But from the relation $ax+\xi^{-2}xa=\Lam^{-1}\theta\xi^{-2}(ba^3-gb)$, we have
\begin{align*}
(a_2+\xi^{-2}a_1)x_3=0,\quad a_1x_2+\xi^{-2}a_2x_2=\Lam^{-1}\theta\xi^{-2}(a_2^3-g_1)b_2,
\end{align*}
which implies $a_2=0=a_1$, a contradiction. Therefore we must have $b_2\neq 0$ and we may also assume $b_2=1$.

The relation $ba=\xi ab$ implies $a_2=\xi a_1$, and the relation $x^2=1-g^2$ implies $x_2x_3=1-g_1^2$. From the relation $bx+\xi^{-2}xb=\theta\xi^{-2}(a^3-g)a$, we have
\begin{align*}
        \left(\begin{array}{ccc}
                                   x_3 & 0\\
                                  0 & \xi^{-2}x_3
                                 \end{array}\right)=\theta\xi^{-2}
        \left(\begin{array}{ccc}
                                   a_1^4-g_1a_1 & 0\\
                                   0  & a_2^4-g_2a_2
                                 \end{array}\right),
\end{align*}
which implies $x_3=\theta\xi^{-2}(a_1^3-g_1)a_1$. Indeed, $\xi^{-2}x_3=\theta\xi^{-2}(a_2^4-g_2a_2)\Leftrightarrow x_3=\theta\xi^{-2}(a_1^3-g_1)a_1$ since $g_1=-g_2$ and $a_2=\xi a_1$. Thus we have $x_2=\theta^{-1}\xi^{2}a_1^{-1}(a_1^3+g_1)$ since $x_2x_3=1-g_1^2=(a_1^3+g_1)(a_1^3-g_1)$.

From the discussion above, the matrices defining the action on $V$ are of the form
\begin{align*}
    [a]_{i,j}&=\left(\begin{array}{ccc}
                                   \Lam_1 & 0\\
                                   0 &    \xi\Lam_1
                                 \end{array}\right),\quad
    [b]_{i,j}=\left(\begin{array}{ccc}
                                   0 & 1\\
                                   0 & 0
                                 \end{array}\right),\quad
    [g]_{i,j}=\left(\begin{array}{ccc}
                                   \Lam_2 & 0\\
                                   0 & -\Lam_2
                                 \end{array}\right),\\
    [x]_{i,j}&=\left(\begin{array}{ccc}
                                   0 & \theta^{-1}\xi^{2}\Lam_1^{-1}(\Lam_1^3+\Lam_2)\\
                                   \theta\xi^{-2}\Lam_1(\Lam_1^3-\Lam_2) & 0
                                 \end{array}\right),
\end{align*}
with $\Lam_1^6=1=\Lam_2^6$. And it is clear that $V$ is simple if and only if $\Lam_1^3-\Lam_2\neq 0$.
If we set $\Lam_1=\xi^i$ and $\Lam_2=\xi^j$ for some $i, j\in Z_6$,
then $3i\neq j$ in $Z_6$, that is, $(i,j)\in\Lambda$ and
in such a case, we denote this module $V$ by $V_{i,j}$.

Now we claim that $V_{i,j}\cong V_{k,l}\Leftrightarrow (i,j)=(k,l)$ in $Z_6\times Z_6$.
Suppose that $\Psi:V_{i,j}\mapsto V_{k,l}$ is a $D$-module isomorphism, and denote by
$[\Psi]=(p_{i,j})_{i,j=1,2}$ the matrix of $\Psi$ in the given basis. As a module morphism, we have $[b][\Psi]=[\Psi][b]$ and $[a][\Psi]=[\Psi][a]$, which imply $p_{21}=0,\,p_{11}=p_{22}$ and $(\xi^k-\xi^i)p_{11}=0,\,(\xi^k-\xi^{i+1})p_{12}=0$. Thus we have $\xi^i=\xi^k$ which yields $p_{12}=0$ since $\Psi$ is an isomorphism.
Similarly, we have $\xi^j=\xi^l$, and then the claim follows.
\end{proof}

\begin{rmk}
For a left $\D$-module $V$, there exists a left dual module denoted by $V\As$ with module structure given by $(h\rightharpoonup f)(v)=f(S(h)\cdot v)$ for all $h\in \D, v\in V, f\in V\As$. A direct computation shows that $V_{i,j}\As\cong V_{-i-1,-j-3}$ for all $(i,j)\in \Lambda$.
\end{rmk}

Finally, we describe all simple left $\D$-modules up to isomorphism.
\begin{thm}\label{thmsimplemoduleD}
There exist $36$ simple left $D$-modules pairwise non-isomorphic, among which $6$ one-dimensional modules are given by Lemma $\ref{onesimple}$ and $30$ two-dimensional simple modules are given by Lemma $\ref{twosimple}$.
\end{thm}
\begin{proof}
Assume that there exists at least one simple module of dimension $x$ bigger than two and denote $n$ the amount of simple modules (up to isomorphism) of dimension $x$. From lemma $\ref{onesimple}$ and Lemma $\ref{twosimple}$, we have that
\begin{align*}
6\times 1^2+30\times 2^2+nx^2=126+nx^2<\dim \D\As=144.
\end{align*}
Then $nx^2<18$, which implies $x=3$ or $4$ and $n=1$. However, in these cases, we must have $6=|G(\D\As)|$ divides $9$ or $16$, a contradiction.
\end{proof}

Now we discuss the projective covers of the simple modules of $\D$. Let $\K_{\chi^j}$ denote the one-dimensional $D$-module associated to the character $\chi^j$ for $j\in Z_6$, $Irr(\D)$ denote the set of isomorphism classes of simple modules and $\Pp(V)$ denote the projective cover of a simple $\D$-module $V$. It is well-known that projective covers are unique up to isomorphism and as a left $\D$-module and one has that
\begin{align*}
\D\cong \oplus_{V\in Irr(D)}\Pp(V)^{\dim\, V}.
\end{align*}

\begin{lem}\label{lemProjectwodimsimple}
\begin{enumerate}
  \item $V_{i,j}\otimes \K_{\chi^{k}}\cong V_{i+k,j+3k}$ and $\K_{\chi^{l}}\otimes\K_{\chi^{k}}\cong \K_{\chi^{k+l}}$ for all $(i,j)\in\Lambda,\,k,\,l\in Z_6$.
  \item $\Pp(V_{i,j})=V_{i,j}$ for all $(i,j)\in\Lambda$.
  \item $\Pp(\K_{\chi^i})=\Pp(\K_{\epsilon})\otimes \K_{\chi^i}$ and $\dim\, \Pp(\K_{\chi^i})=4$ for all $i\in Z_6$.
\end{enumerate}
\end{lem}
\begin{proof}
\begin{enumerate}
  \item follows by a direct calculation.
  \item For any fixed $(i,j)\in\Lambda$. Note that
  \begin{align*}
   \hom(\Pp(V_{i,j})\otimes\K_{\chi^k},V_{i,j}\otimes\K_{\chi^k})&\cong \hom(\Pp(V_{i,j}),V_{i,j}\otimes\K_{\chi^k}\otimes\K_{\chi^k}\As)\\&\cong \hom(\Pp(V_{i,j}),V_{i,j})\neq 0.
   \end{align*}
    Since $\Pp(V_{i,j})\otimes\K_{\chi^k}$ is projective, $\Pp(V_{i,j})\otimes\K_{\chi^k}$ must contain
      $\Pp(V_{i,j}\otimes\K_{\chi^k})\cong \Pp(V_{i+k,j+3k})$. If $\dim\,\Pp(V_{i,j})>\dim\,V_{i,j}$ for some $(i,j)\in\Lambda$, then the socle of $\Pp(V_{i,j})$ is $V_{i,j}$. Thus $\dim\,\Pp(V_{i,j})\geq 2 \dim\,V_{i,j}$ and $\dim\Pp(V_{i-k,j-3k})\geq \dim\Pp(V_{i,j})\geq 4$. Now if we denote $I=\{(m,n)\in\Lambda\mid(m,n)\neq (i+k,j+3k)$ for all $k\in Z_6\}$. Then
      \begin{align*}
      \dim \D=\sum_{i=0}^{5}\dim\, \Pp(\K_{\chi^i})+\sum_{(m.n)\in I}2\dim \Pp(V_{m,n})+ 12\dim\, \Pp(V_{i,j})>144.
      \end{align*}
      It is a contradiction since $\dim\, \D=144$ and then claim follows.
  \item The proof is similar to the above proof.
\end{enumerate}
\end{proof}
\begin{rmk}\label{rmkSocToponedim}
If the simple modules $V_{i,j}$ with $(i,j)\in\Lambda$ are projective, then any $V_{i,j}$ cannot be contained in the socle or the Top of any non-simple indecomposable module, which implies that the Top and the socle of any non-simple indecomposable module consist of direct sums of one-dimensional modules by Theorem $\ref{thmsimplemoduleD}$ and Lemma $\ref{lemProjectwodimsimple}$.
\end{rmk}
Now we try to describe the $\Pp(\K_{\epsilon})$.

Let $\Pp$ be a left $\D$-module. The matrices defining the action on a given basis $\{p_i\}_{1\leq i\leq 4}$  are of the form
\begin{align}
\begin{split}\label{eqprojective0}
    [a]&=\left(\begin{array}{cccc}
                                    1 & 0   &  0  & 0\\
                                    0 & \xi &  0   & 0\\
                                    0 & 0   &  \xi^{-1} &0 \\
                                    0 & 0 & 0 & 1
                                 \end{array}\right),\quad
   [b]=\left(\begin{array}{cccc}
                                    0 & 0   &  0  & 0\\
                                    0 & 0 &  0   & 0\\
                                    \theta & 0   &  0 &0 \\
                                    0 & 1 & 0 & 0
                                 \end{array}\right),\\
   [g]&=\left(\begin{array}{cccc}
                                    1 & 0   &  0  & 0\\
                                    0 & -1 &  0   & 0\\
                                    0 & 0   &  -1 &0 \\
                                    0 & 0 & 0 & 1
                                 \end{array}\right),\quad
   [x]=\left(\begin{array}{cccc}
                                    0 & 0   &  0  & 0\\
                                    \theta & 0 &  0   & 0\\
                                    2 & 0   &  0 &0 \\
                                    0 & 2(1+\xi)\theta & \xi^2 & 0
                                 \end{array}\right).
\end{split}
\end{align}

\begin{lem}\label{lemprojectivecover}
$\Pp(\K_{\epsilon})\cong \Pp$ as $\D$-modules.
\end{lem}
\begin{proof}
A direct computation shows that $\Pp$ as a $\D$-module is well-defined, $\K\{p_4\}$ $\cong \K_{\epsilon}$ and $\Pp/Q\cong \K_{\epsilon}$ as $\D$-module, where $Q=\{p_2,p_3,p_4\}$.

We first claim that $\Pp$ is indecomposable. Assume that $\Pp=A\oplus B$ as $\D$-module and let $\alpha=\alpha_1p_1+\alpha_2p_2+\alpha_3p_3+\alpha_4p_4\in A$. If $\alpha_1\neq 0$, then $(xb)\cdot\alpha=\xi^2\theta\alpha_1p_4\in A$ and thus $p_4\in A$. Analogously, $b\cdot\alpha=\theta\alpha_1p_3+\alpha_2p_4$ which implies $p_3\in A$. Thus $\alpha_1p_1+\alpha_2p_2\in A$. since $x\cdot(\alpha_1p_1+\alpha_2p_2)=\theta\alpha_1p_2+2\alpha_1p_3+2(1+\xi)\theta\alpha_2p_4\in A$, we have $p_2\in A$ and consequently $p_1\in A$, which implies $A=\Pp$. If $\alpha_1=0$, then $p_1\in B$ and we have $B=\Pp$ since $\Pp$ can be generated by $p_1$. Hence $\Pp$ is indecomposable.

Now we try to show that $\Pp(\K_{\epsilon})\cong \Pp$. Denote $p:\Pp\mapsto \Pp/Q\cong \K_{\epsilon}$ and it is clear that $p$ is a surjective. Since $\Pp(\K_{\epsilon})$ is projective, there exists an $\D$-module morphism denoted by $f$ such that the following diagram commutes
\begin{align*}
\xymatrix{
  \Pp(\K_{\epsilon}) \ar[d]_{f} \ar[dr]^{\pi}        \\
  \Pp \ar[r]_{p}  & \K_{\epsilon}              }
\end{align*}
Thus there exists some $\alpha=\alpha_1p_1+\alpha_2p_2+\alpha_3p_3+\alpha_4p_4\in f(\Pp(\K_{\epsilon}))$ with $\alpha_1\neq 0$ which implies $f$ is surjective and whence an isomorphism.
\end{proof}

From the lemma $\ref{lemProjectwodimsimple}$, we have known $\Pp(\K_{\chi^i})=\Pp(\K_{\epsilon})\otimes \K_{\chi^i}$.
Then as a result of lemma $\ref{lemprojectivecover}$, we can describe the module structure of $\Pp(\K_{\chi^i})$ explicitly.
\begin{pro}
The $\D$-modules $\Pp_i=\Pp\otimes \Pp(\K_{\epsilon})$ with $i\in Z_6$ and $V_{j,k}$ with $(j,k)\in\Lambda$ are the projective covers of the simple $\D$-modules $\K_{\chi^i}$ and $V_{j,k}$ respectively. Moreover, for $j\in Z_6$, $\{p_{i,j}\}_{1\leq i\leq 4}$ is a linear basis of $\Pp_j$ with $p_{i,0}=p_i$, and the $\D$-module structure of $\Pp_j$ can be given explicitly by
\begin{align}
\begin{split}\label{eqprojectivei}
a\cdot p_{i,j}&=a\cdot(p_i\otimes 1)=a\cdot p_i\otimes a\cdot 1+\Lam^{-1}ba^3\cdot p_i\otimes b\cdot 1=\xi^j(a\cdot p_i)\otimes 1,\\
b\cdot p_{i,j}&=b\cdot(p_i\otimes 1)=b\cdot p_i\otimes a\cdot 1+a^4\cdot p_i\otimes b\cdot 1=\xi^j(b\cdot p_i)\otimes 1,\\
g\cdot p_{i,j}&=g\cdot(p_i\otimes 1)=g\cdot p_i\otimes g\cdot 1=(-1)^j(g\cdot p_i)\otimes 1,\\
x\cdot p_{i,j}&=x\cdot(p_i\otimes 1)=1\cdot p_i\otimes x\cdot 1+x\cdot p_i\otimes g\cdot 1=(-1)^j(x\cdot p_i)\otimes 1.
\end{split}
\end{align}
\end{pro}

Now we start to study the representation type of the Drinfeld double $\D(\C^{cop})$. To do so, we begin to find
$2$-dimensional non-simple indecomposable modules.

 Consider the subalgebra $A$ of $\D$ generated by the elements $a$ and $g$. It is clear that $A$ is a finite-dimensional commutative algebra, and it is simple modules are all one-dimensional. In particular, the restriction to $A$ of characters of $\D$ can induce characters on $A$.

Let $M$ be any $2$-dimensional non-simple indecomposable $D$-module contained $\K_{\Lam}$ with $\Lam=\chi^l$. We must have that $M\cong \K_{\Lam}\oplus\K_{\mu}$ as $A$-modules, with $\mu$ some character on $\D$. Thus, $M$ has a linear basis $\{m_1, m_2\}$, such that $\K m_1\cong \K_{\Lam}$, $a\cdot m_2=\mu(a)m_2,\,g\cdot m_2=\mu(g)m_2$, and fits into an exact sequence
\begin{align*}
0\rightarrow \K_{\Lam}\hookrightarrow M\rightarrow \K_{\mu}\twoheadrightarrow 0.
\end{align*}
Then we must have that $b\cdot m_2= \alpha m_1,\ x\cdot m_2=\beta m_1$ for some $\alpha, \beta\in\K$.
As above, given the basis $\{m_1, m_2\}$, the matrices defining $\D$-action on $M$ are of the form
\begin{align*}
    [a]&=\left(\begin{array}{ccc}
                                   \Lam(a) & 0\\
                                    0  & \mu(a)
                                 \end{array}\right),\quad
    [b]=\left(\begin{array}{ccc}
                                   0 & \alpha\\
                                   0 & 0
                                 \end{array}\right),\\
    [g]&=\left(\begin{array}{ccc}
                                   \Lam(g) & 0\\
                                   0   & \mu(g)
                                 \end{array}\right),\quad
    [x]=\left(\begin{array}{ccc}
                                   0 & \beta\\
                                   0 & 0
                                 \end{array}\right).
\end{align*}
From the relations $gx=-xg$ and $gb=-bg$, we have that $(\Lam(g)+\mu(g))\beta=0$ and $(\Lam(g)+\mu(g))\alpha=0$. If $\Lam(g)+\mu(g)\neq 0$, then $\beta=0=\alpha$. In such a case, it is clear that $M\cong \K_{\Lam}\oplus\K_{\mu}$ as $\D$-modules, a contradiction. Thus $\Lam(g)+\mu(g)= 0$. Moreover, from the relation $bx+\xi^{-2}xb=\theta\xi^{-2}(a^4-ga)$, we have $\Lam(a^3)=\Lam(g)$ and $\mu(a^3)=\mu(g)$, which imply $\mu(a^3)=-\Lam(a^3)$. And from the relation $ax+\xi^{-2}xa=\Lam^{-1}\theta\xi^{-2}(ba^3-gb)$, we have $(\Lam(a)+\xi^{-2}\mu(a))\beta=\Lam^{-1}\theta\xi^{-2}(\mu(a^3)-\Lam(g))\alpha$ which implies
\begin{align}\label{equtwononsimple}
(\Lam(a)+\xi^{-2}\mu(a))\beta=2\Lam^{-1}\theta\xi^{-2}\mu(a^3)\alpha,
\end{align}
since $\mu(a^3)=-\Lam(g)$.

From the discussion above, $\mu(a)=\xi\Lam(a),\,\xi^3\Lam(a)$ or $\xi^5\Lam(a)$.
If $\mu(a)=\xi\Lam(a)$,  then by the equation $\eqref{equtwononsimple}$, we have $-\theta^3\beta=2\xi\Lam(a^2)\alpha$. We may assume that $\alpha=-\theta^3$ and $\beta=2\xi\Lam(a^2)$. In such a case, we denote the module by $M_l^{+}$, and we have the following result
\begin{align*}
    [a]_l^{+}&=\left(\begin{array}{ccc}
                                   \chi^l(a) & 0\\
                                    0  & \chi^{l+1}(a)
                                 \end{array}\right),\quad
    [b]_l^{+}=\left(\begin{array}{ccc}
                                   0 & -\theta^3\\
                                   0 & 0
                                 \end{array}\right),\\
    [g]_l^{+}&=\left(\begin{array}{ccc}
                                   \chi^l(g) & 0\\
                                   0   & \chi^{l+1}(g)
                                 \end{array}\right),\quad
    [x]_l^{+}=\left(\begin{array}{ccc}
                                   0 & 2\xi\chi^l(a^2)\\
                                   0 & 0
                                 \end{array}\right),
\end{align*}
\begin{align*}
0\rightarrow \K_{\chi^l}\hookrightarrow M^{+}\rightarrow \K_{\chi^{l+1}}\twoheadrightarrow 0.
\end{align*}
If $\mu(a)=\xi^3\Lam(a)$, then by the equation $\eqref{equtwononsimple}$, we have $\theta\beta=2\xi\Lam(a^2)\alpha$. We may assume that $\alpha=\theta$ and $\beta=2\xi\Lam(a^2)$. In such a case, we denote the module by $M_l^{\pm}$, and we have the following result
\begin{align*}
    [a]_l^{\pm}&=\left(\begin{array}{ccc}
                                   \chi^l(a) & 0\\
                                    0  & \chi^{l\pm 3}(a)
                                 \end{array}\right),\quad
    [b]_l^{\pm}=\left(\begin{array}{ccc}
                                   0 & \theta\\
                                   0 & 0
                                 \end{array}\right),\\
    [g]_l^{\pm}&=\left(\begin{array}{ccc}
                                   \chi^l(g) & 0\\
                                   0   & \chi^{l\pm 3}(g)
                                 \end{array}\right),\quad
    [x]_l^{\pm}=\left(\begin{array}{ccc}
                                   0 & 2\xi\chi^l(a^2),\\
                                   0 & 0
                                 \end{array}\right),
\end{align*}
\begin{align*}
0\rightarrow \K_{\chi^l}\hookrightarrow M^{\pm}\rightarrow \K_{\chi^{l\pm 3}}\twoheadrightarrow 0.
\end{align*}
If $\mu(a)=\xi^5\Lam(a)$, then by the equation $\eqref{equtwononsimple}$, we have $\alpha=0$. We may assume that $\beta=1$. In such a case, we denote the module by $M_{l}^{-}$, and we have the following result
\begin{align*}
    [a]_l^{-}&=\left(\begin{array}{ccc}
                                   \chi^l(a) & 0\\
                                    0  & \chi^{l-1}(a)
                                 \end{array}\right),\quad
    [b]_l^{-}=\left(\begin{array}{ccc}
                                   0 & 0\\
                                   0 & 0
                                 \end{array}\right),\\
    [g]_l^{-}&=\left(\begin{array}{ccc}
                                   \chi^l(g) & 0\\
                                   0   & \chi^{l-1}(g)
                                 \end{array}\right),\quad
    [x]_l^{-}=\left(\begin{array}{ccc}
                                   0 & 1\\
                                   0 & 0
                                 \end{array}\right),
\end{align*}
\begin{align*}
0\rightarrow \K_{\chi^l}\hookrightarrow M^{-}\rightarrow \K_{\chi^{l-1}}\twoheadrightarrow 0.
\end{align*}

By the preceding discussion, we have the following result:
\begin{lem}\label{lemTwononsimpleindecom}
\begin{enumerate}
  \item Let $M$ be a $2$-dimensional non-simple indecomposable module containing $\K_{\chi^{\ell}}$. Then
        $M\cong M_l^{+}$, $M\cong M_l^{-}$, or $M\cong M_l^{\pm}$.
  \item
  \begin{align*}
         \dim \,Ext_{\D}^1(\K_{\chi^i},\K_{\chi^j})=
        \begin{cases}
               1, &\text{~if~} i=j\pm 1,\text{~or~}i=j\pm 3,\\
               0, & \text{otherwise}.
        \end{cases}
        \end{align*}
\end{enumerate}
\end{lem}
\begin{rmk}\label{rmksocletopTwoindec}
Note that $Soc(M_l^{+})=\K_{\chi^{l}}$, $Top(M_l^{+})=\K_{\chi^{l+1}}$, $Soc(M_l^{\pm})=\K_{\chi^{l}}$, $Top(M_l^{\pm})=\K_{\chi^{l\pm 3}}$, and $Soc(M_l^{-})=\K_{\chi^{l}}$, $Top(M_l^{-})=\K_{\chi^{l-1}}$.
\end{rmk}

By Lemma $\ref{lemProjectwodimsimple}$, we know that $V_{i,j}$ for any $(i,j)\in\Lambda$ is projective, thus we can get following lemma and as a corollary, $\D$ is of wild representation type.
\begin{lem}\label{lemExtproperty}
\begin{enumerate}
  \item $\dim\,Ext_{\D}^1(V_{i,j}, V_{k,\ell})=0$ for all $(i,j),(k,\ell)\in\Lambda$.
  \item $\dim\,Ext_{\D}^1(V_{i,j}, \K_{\chi^{\ell}})=0$ and $\dim\, Ext_{\D}^1(\K_{\chi^{\ell}},V_{i,j})=0$ for all $(i,j)\in\Lambda$, and $\ell\in Z_6$.
\end{enumerate}
\end{lem}

\begin{cor}
$\D$ is of wild representation type.
\end{cor}
\begin{proof}
From the Lemma $\ref{lemExtproperty}$, we know that $ExtQ(\D)$ contains the quiver
$$
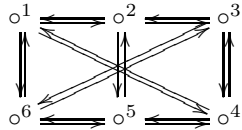
\xymatrix{
{\circ}^1\ar@{<-}[r]\ar@<-0.6mm>[r]\ar@{<-}[1,2]\ar@<-0.6mm>[1,2]  & {\circ}^{2}\ar@{<-}[r]\ar@<-0.6mm>[r]\ar@<-1.0mm>[d]  &  {\circ}^{3}\ar@{<-}[d]\ar@{<-}[1,-2]\ar@<-0.8mm>[1,-2]\ar@<-0.8mm>[d]\\
{\circ}^{6}\ar@{<-}[r]\ar@<-0.6mm>[u]\ar@{<-}[u]\ar@<-0.6mm>[r]  & {\circ}^{5}\ar@{<-}[r]\ar@{->}[u]\ar@<-0.6mm>[r]             & {\circ}^{4}
}
$$
where the vertex $i$ represents the one-dimensional simple module $\K_{\chi^{i}}$ for $i\in Z_6$. Thus the separation diagram of $\D$ contains the quivers as follows
$$
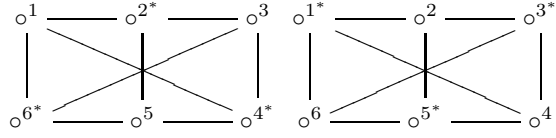
\xymatrix{
{\circ}^1\ar@{-}[r]\ar@{-}[1,2]  & {\circ}^{2\As}\ar@{-}[r]  &  {\circ}^{3}\ar@{-}[d]\ar@{-}[1,-2]\\
{\circ}^{6\As}\ar@{-}[r]\ar@{-}[u]              & {\circ}^{5}\ar@{-}[r]\ar@{-}[u]             & {\circ}^{4\As}
}
\xymatrix{
{\circ}^{1\As}\ar@{-}[r]\ar@{-}[1,2]  & {\circ}^{2}\ar@{-}[r]  &  {\circ}^{3\As}\ar@{-}[d]\ar@{-}[1,-2]\\
{\circ}^{6}\ar@{-}[r]\ar@{-}[u]              & {\circ}^{5\As}\ar@{-}[r]\ar@{-}[u]             & {\circ}^{4}
}
$$
Then by Proposition $\ref{proRepType}$, $\D$ is of wild type.
\end{proof}

\section{Yetter-Drinfeld modules category $\CYD$}\label{secCYD}
In this section, we determine the simple objects and projective covers of simple modules of $\CYD$ by using the monoidal category equivalence $\DM\simeq \CYD$, and also describe their braiding. Indeed, the braiding of $2$-dimensional simple modules are triangular \cite{U07}, and after a tedious calculation similar to \cite[Appendix]{GG16}, the braiding is not of diagonal type. To do so, we first need to describe the coaction of $\C$.

\begin{pro}
Let $\K_{\chi^i}=\K v$ be a one-dimensional $\D$-module with $ i\in Z_6$. Then $\K_{\chi^i}\in\CYD$ with the module structure and comodule structure given by
\begin{align*}
a\cdot v=\xi^i v,\quad b\cdot v=0,\quad \delta(v)=a^{3i}\otimes v.
\end{align*}
\end{pro}
\begin{proof}
Since $\K_{\chi^i}=\K v$ is a one dimensional $\D$-module with $ i\in Z_6$, the $\C$-action must be given by the restriction of the character of $\D$ given by lemma $\ref{onesimple}$ and the coaction must be of the form $\delta(v)=h\otimes v$ where $h\in G(\C)=\{1,a^3\}$ such that $\langle g, h\rangle v=(-1)^iv$. It follows that the action is given by $a\cdot v=\xi^i v,\, b\cdot v=0$ and the coaction is given by $\delta(v)=a^{3i}\otimes v$.
\end{proof}
The following Proposition gives the braiding of $\K_{\chi^i}$ for all $i\in Z_6$.
\begin{pro}\label{braidingone}
The braiding of the one-dimensional YD-module $\K_{\chi^i}=\K v$ is $c(v\otimes v)=(-1)^iv\otimes v$.
\end{pro}

\begin{pro}
Let $V_{i,j}=\K\{v_1,v_2\}$ be a two-dimensional simple $\D$-module with $(i,j)\in\Lambda$. If we denote $\Lam_1=\xi^i$ and $\Lam_2=\xi^j$, then $V_{i,j}\in\CYD$ with the module structure  given by
\begin{align*}
a\cdot v_1=\Lam_1v_1,\quad b\cdot v_1=0,\quad a\cdot v_2=\xi\Lam_1 v_2,\quad b\cdot v_2=v_1,
\end{align*}
and the comodule structure given by

\medskip\noindent
for $j=0$, $i\in\{1,3,5\}:$
\begin{align*}
\delta(v_1)&=1\otimes v_1+2\xi\Lam_1ba^5\otimes v_2,\quad\delta(v_2)=a^3\otimes v_2,\\
\intertext{for $j=1$, $i\in Z_6:$}
\delta(v_1)&=a^5\otimes v_1+(\xi^4\Lam_1^4-\xi^5\Lam_1)ba^4\otimes v_2,\,
\delta(v_2) =a^2\otimes v_2+(\Lam_1^2+\xi\Lam_1^{-1})ba\otimes v_1,\\
\intertext{for $j=2$, $i\in Z_6:$}
\delta(v_1)&=a^4\otimes v_1+(\xi^4\Lam_1^4-\Lam_1)ba^3\otimes v_2,\,
\delta(v_2)=a\otimes v_2+(\Lam_1^2+\xi^2\Lam_1^{-1})b\otimes v_1,\\
\intertext{for $j=3$, $i\in\{0,2,4\}:$}
\delta(v_1)&=a^3\otimes v_1+(\xi^4\Lam_1^4-\xi\Lam_1)ba^2\otimes v_2, \quad\delta(v_2)=1\otimes v_2,\\
\intertext{for $j=4$, $i\in Z_6:$}
\delta(v_1)&=a^2\otimes v_1+(\xi^4\Lam_1^4-\xi^2\Lam_1)ba\otimes v_2,\,
\delta(v_2)=a^5\otimes v_2+(\Lam_1^2+\xi^{4}\Lam_1^{-1})ba^4\otimes v_1,\\
\intertext{for $j=5$, $i\in Z_6:$}
\delta(v_1)&=a\otimes v_1+(\xi^4\Lam_1^4-\xi^3\Lam_1)b\otimes v_2 ,\,
\delta(v_2)=a^4\otimes v_2+(\Lam_1^2+\xi^{5}\Lam_1^{-1})ba^3\otimes v_1.
\end{align*}
\end{pro}
\begin{proof}
Note that by Lemma $\ref{lemAtoCdual}$ and Remark $\ref{rmkAtoCdual}$, we have
\begin{align*}
(g^i)\As=\frac{1}{6}(1+\xi^ia+\xi^{2i}a^2+\xi^{3i}a^3+\xi^{4i}a^4+\xi^{5i}a^5),\\
(g^ix)\As=\frac{1}{6\theta}(\xi^ib+\xi^{2i}ba+\xi^{3i}ba^2+\xi^{4i}ba^3+\xi^{5i}ba^4+ba^5).
\end{align*}
Denote by $\{c_i\}_{1\leq i\leq 12}$ and $\{c^i\}_{1\leq i\leq 12}$ a basis of $\C$ and its dual basis respectively. Then the comodule structure is given by $\delta(v)=\sum_{i=1}^{12}c_i\otimes c^i\cdot v$ for any $v\in V_{i,j}$. Thus
\begin{align*}
\delta(v_1)&=\sum_{k=0}^{5}(g^k)\As\otimes g^k\cdot v_1+\sum_{k=0}^{5}(g^kx)\As\otimes g^kx\cdot v_1\\
           &=\sum_{k=0}^{5}\Lam_2^k(g^k)\As\otimes v_1+\sum_{k=0}^{5}(\Lam_2)^k(g^kx)\As\otimes  x_2v_2,\\
\delta(v_2)&=\sum_{k=0}^{5}(g^k)\As\otimes g^k\cdot v_2+\sum_{k=0}^{5}(g^kx)\As\otimes g^kx\cdot v_2\\
           &=\sum_{k=0}^{5}(-\Lam_2)^k(g^k)\As\otimes v_2+\sum_{k=0}^{5}(-\Lam_2)^k(g^kx)\As\otimes  x_1v_1,
\end{align*}
where $x_1=\theta^{-1}\xi^{2}\Lam_1^{-1}(\Lam_1^3+\Lam_2)$ and $x_2=\theta\xi^{-2}\Lam_1(\Lam_1^3-\Lam_2)$. Now we describe the coaction explicitly case by case.

If $j=0$, then $i\in\{1,3,5\}$. In such a case, $\Lam_2=1$, $x_1=0$, $x_2=2\theta\xi\Lam_1=2\theta\xi^{i+1}$ and then $\delta(v_1)=1\otimes v_1+2\xi\Lam_1ba^5\otimes v_2$ and $\delta(v_2)=a^3\otimes v_2$.

If $j=1$, then $i\in Z_6$. In such a case, $\Lam_2=\xi$, $x_1=\theta^{-1}\xi^{2}(\Lam_1^2+\xi\Lam_1^{-1})$, $x_2=\theta(\xi^4\Lam_1^4-\xi^5\Lam_1)$ and then
$\delta(v_1)=a^5\otimes v_1+(\xi^4\Lam_1^4-\xi^5\Lam_1)ba^4\otimes v_2$
and
$\delta(v_2) =a^2\otimes v_2+(\Lam_1^2+\xi\Lam_1^{-1})ba\otimes v_1$.

If $j=2$, then $i\in Z_6$. In such a case, $\Lam_2=\xi^2$, $x_1=\theta^{-1}\xi^{2}(\Lam_1^2+\xi^2\Lam_1^{-1})$, $x_2=\theta(\xi^4\Lam_1^4-\Lam_1)$ and then
$\delta(v_1)=a^4\otimes v_1+(\xi^4\Lam_1^4-\Lam_1)ba^3\otimes v_2,\,
\delta(v_2)=a\otimes v_2+(\Lam_1^2+\xi^2\Lam_1^{-1})b\otimes v_1$.

If $j=3$, then $i\in \{0,2,4\}$. In such a case, $\Lam_2=\xi^3$, $x_1=0$, $x_2=\theta\xi^{-2}(\Lam_1^4+\Lam_1)$ and then
$\delta(v_1)=a^3\otimes v_1+\xi^{-2}(\Lam_1^4+\Lam_1)ba^2\otimes v_2$
and
$\delta(v_2)=1\otimes v_2$.

If $j=4$, then $i\in Z_6$. In such a case, $\Lam_2=\xi^4$, $x_1=\theta^{-1}\xi^{2}(\Lam_1^2+\xi^{4}\Lam_1^{-1})$, $x_2=\theta(\xi^4\Lam_1^4-\xi^2\Lam_1)$ and then
$\delta(v_1)=a^2\otimes v_1+(\xi^4\Lam_1^4-\xi^2\Lam_1)ba\otimes v_2,\,
\delta(v_2)=a^5\otimes v_2+(\Lam_1^2+\xi^{4}\Lam_1^{-1})ba^4\otimes v_1$ .

If $j=5$, then $i\in Z_6$. In such a case, $\Lam_2=\xi^5$, $x_1=\theta^{-1}\xi^{2}(\Lam_1^2+\xi^{5}\Lam_1^{-1})$, $x_2=\theta(\xi^4\Lam_1^4-\xi^3\Lam_1)$ and then
$\delta(v_1)=a\otimes v_1+(\xi^4\Lam_1^4-\xi^3\Lam_1)b\otimes v_2 ,\,
\delta(v_2)=a^4\otimes v_2+(\Lam_1^2+\xi^{5}\Lam_1^{-1})ba^3\otimes v_1$ .
\end{proof}

Now using the braiding in $\CYD$ (see equation $\eqref{equbraidingYDcat}$),  we describe the braiding of the simple modules $V_{i,j}\in\CYD$.
\begin{pro}\label{probraidsimpletwo}
Let $V_{i,j}=\K\{v_1,v_2\}$ be a two-dimensional simple $D$-module with $(i,j)\in\Lambda$. If we denote $\Lam_1=\xi^i$ and $\Lam_2=\xi^j$, then $V_{i,j}\in\CYD$. The braiding of $V_{i,j}=\K\{v_1,v_2\}$ is given by
\begin{enumerate}
  \item If $j=0$, $i\in\{1,3,5\}$,
  \begin{align*}
   c(\left[\begin{array}{ccc} v_1\\v_2\end{array}\right]\otimes\left[\begin{array}{ccc} v_1~v_2\end{array}\right])=
   \left[\begin{array}{ccc}
                     v_1\otimes v_1    & v_2\otimes v_1+2v_1\otimes v_2\\
                     -v_1\otimes v_2   & v_2\otimes v_2
         \end{array}\right].
  \end{align*}
  \item If $j=1$, $i\in Z_6$,
  \begin{align*}
   c(\left[\begin{array}{ccc} v_1\\v_2\end{array}\right]\otimes\left[\begin{array}{ccc} v_1~v_2\end{array}\right])=
   \left[\begin{array}{ccc}
       \Lam_1^5v_1\otimes v_1    & \xi^5\Lam_1^5v_2\otimes v_1{+}(\Lam_1^5{+}\Lam_1^2\xi^2)v_1\otimes v_2\\
       \Lam_1^2v_1\otimes v_2            & \xi^2\Lam_1^2v_2\otimes v_2{+}(\xi\Lam_1^3{+}\xi^2)v_1\otimes v_1
         \end{array}\right].
  \end{align*}
  \item If $j=2$, $i\in Z_6$,
  \begin{align*}
   c(\left[\begin{array}{ccc} v_1\\v_2\end{array}\right]\otimes\left[\begin{array}{ccc} v_1~v_2\end{array}\right])=
   \left[\begin{array}{ccc}
       \Lam_1^4v_1\otimes v_1    & \xi^4\Lam_1^4v_2\otimes v_1{+}(\xi\Lam_1{+}\Lam_1^4)v_1\otimes v_2\\
       \Lam_1v_1\otimes v_2            & \xi\Lam_1v_2\otimes v_2{+}(\Lam_1^2{+}\xi^2\Lam_1^5)v_1\otimes v_1
         \end{array}\right].
  \end{align*}
  \item If $j=3$, $i\in\{0,2,4\}$,
  \begin{align*}
   c(\left[\begin{array}{ccc} v_1\\v_2\end{array}\right]\otimes\left[\begin{array}{ccc} v_1~v_2\end{array}\right])=
   \left[\begin{array}{ccc}
                     v_1\otimes v_1    & -v_2\otimes v_1+2v_1\otimes v_2\\
                     v_1\otimes v_2   & v_2\otimes v_2
         \end{array}\right].
  \end{align*}
  \item If $j=4$, $i\in Z_6$,
  \begin{align*}
   c(\left[\begin{array}{ccc} v_1\\v_2\end{array}\right]\otimes\left[\begin{array}{ccc} v_1~v_2\end{array}\right])=
   \left[\begin{array}{ccc}
       \Lam_1^2v_1\otimes v_1    & \xi^2\Lam_1^2v_2\otimes v_1{+}(\xi^5\Lam_1^5{+}\Lam_1^2)v_1\otimes v_2\\
       \Lam_1^5v_1\otimes v_2            & \xi^5\Lam_1^5v_2\otimes v_2{+}(\xi^2\Lam_1^3{+}\xi^4)v_1\otimes v_1
         \end{array}\right].
   \end{align*}
  \item If $j=5$, $i\in Z_6$,
  \begin{align*}
   c(\left[\begin{array}{ccc} v_1\\v_2\end{array}\right]\otimes\left[\begin{array}{ccc} v_1~v_2\end{array}\right])=
   \left[\begin{array}{ccc}
       \Lam_1v_1\otimes v_1    & \xi\Lam_1v_2\otimes v_1{+}(\xi^4\Lam_1^4{+}\Lam_1)v_1\otimes v_2\\
       \Lam_1^4v_1\otimes v_2  & \xi^4\Lam_1^4v_2\otimes v_2{+}(\xi^3\Lam_1^5{+}\xi^2\Lam_1^{2}) v_1\otimes v_1
       \end{array}\right].
   \end{align*}
\end{enumerate}
\end{pro}

 Similarly, we give the description of the projective covers of the one dimensional modules $\K_{\chi^j}$ for $k\in Z_6$ as objects in $\CYD$.
\begin{pro}
Let $\Pp_j$ be the projective cover of the one-dimensional $\D$-module $\K_{\chi^j}$ for $j\in Z_6$. Then $\Pp_j\in\CYD $ with its module structure given by $\eqref{eqprojective0}$ and $\eqref{eqprojectivei}$ and its comodule structure given by
\begin{align*}
\delta(p_{1,j})&=(a^3)^j\otimes p_{1,j}+\theta^{-1}(-1)^jba^5(a^3)^j\otimes (\theta p_{2,j}+2p_{3,j}),\\
\delta(p_{2,j})&=(a^3)^{j+1}\otimes p_{2,j}+2(1+\xi)(-1)^jba^2(a^3)^j\otimes p_{4,j},\\
\delta(p_{3,j})&=(a^3)^{j+1}\otimes p_{3,j}+\xi^2\theta^{-1}(-1)^jba^2(a^3)^j\otimes p_{4,j},\\
\delta(p_{4,j})&=(a^3)^j\otimes p_{4,j}.
\end{align*}
\end{pro}

Finally, we describe the braiding of $\Pp_j\in\CYD$.
\begin{pro}\label{probraidingprocover}
The braiding of $\Pp_j$ is given by
\begin{align*}
   c(p_{1,j}\otimes\left[\begin{array}{ccc} p_{1,j}\\p_{2,j}\\p_{3,j}\\p_{4,j}\end{array}\right])&=
   \left[\begin{array}{ccc} (-1)^jp_{1,j} \\p_{2,j}\\p_{3,j}\\(-1)^jp_{4,j}\end{array}\right]\otimes p_{1,j}+
   \left[\begin{array}{ccc}-p_{3,j} \\ (-1)^j\theta^{-1}\xi^{5}p_{4,j}\\0\\0\end{array}\right]\otimes (\theta p_{2,j}{+}2p_{3,j}),\\
   c(p_{2,j}\otimes\left[\begin{array}{ccc} p_{1,j}\\p_{2,j}\\p_{3,j}\\p_{4,j}\end{array}\right])&=
   \left[\begin{array}{ccc}p_{1,j} \\(-1)^{j+1}p_{2,j}\\(-1)^{j+1}p_{3,j}\\p_{4,j}\end{array}\right]\otimes p_{2,j}+
   \left[\begin{array}{ccc}2(-1)^{j+1}(1+\xi)\theta p_{3,j}\\2(1+\xi^5)p_{4,j}\\0\\0\end{array}\right]\otimes p_{4,j},\\
   c(p_{3,j}\otimes\left[\begin{array}{ccc} p_{1,j}\\p_{2,j}\\p_{3,j}\\p_{4,j}\end{array}\right])&=
   \left[\begin{array}{ccc}p_{1,j} \\(-1)^{j+1}p_{2,j}\\(-1)^{j+1}p_{3,j}\\p_{4,j}\end{array}\right]\otimes p_{3,j}+
   \left[\begin{array}{ccc}(-1)^{j+1}\xi^2 p_{3,j}\\\theta^{-1}\xi p_{4,j}\\0\\0\end{array}\right]\otimes p_{4,j},\\
   c(p_{4,j}\otimes\left[\begin{array}{ccc} p_{1,j}\\p_{2,j}\\p_{3,j}\\p_{4,j}\end{array}\right])&=
   \left[\begin{array}{ccc}(-1)^jp_{1,j} \\p_{2,j}\\p_{3,j}\\(-1)^jp_{4,j}\end{array}\right]\otimes p_{4,j}.
\end{align*}
\end{pro}

\section{Nichols algebras in $\CYD$}\label{secNicholsalg}
In this section, we show that the Nichols algebras $\BN(\K_{\chi^{k}})$ for $k\in\{1,3,5\}$ are finite-dimensional exterior algebras and $\BN(V)$ is finite-dimensional if $V$ is isomorphic either to $V_{3,1}$, $V_{3,5}$, $V_{2,2}$ or $V_{2,4}$, which as an algebra is isomorphic to quantum linear space but as a coalgebra is more complicated.
First, we study the Nichols algebras of the one-dimensional simple modules and their projective covers.

By Proposition $\ref{braidingone}$, the following result follows immediately.
\begin{lem}\label{lemNicholsgeneratedbyone}
The Nichols algebra $\BN(\K_{\chi^k})$ associated to $\K_{\chi^k}=\K v$ are
\begin{align*}
\BN(\K_{\chi^k})=\begin{cases}
\K v &\text{~if~} k=0,\,2,\,4,\\
\bigwedge \K_{\chi^k} &\text{~if~} k=1,\,3,\,5.\\
\end{cases}
\end{align*}
Moreover, let $V=\oplus_{i\in I}V_i$, where $V_i\cong \K_{\chi^{k_i}}$ with $k_i\in\{1,3,5\}$, and $I$ is a finite index set. Then $\BN(V)=\bigwedge V\cong \oplus_{i\in I}\BN(V_i)$.
\end{lem}

\begin{lem}
Let $\Pp_j$ be the projective cover of the one-dimensional $D$-module $\K_{\chi^j}$ for $j\in Z_6$. Then $dim\, \BN(\Pp_j)=\infty$.
\end{lem}
\begin{proof}
By Proposition $\ref{probraidingprocover}$, the braiding of $\Pp_j$ for all $j\in Z_6$ has an eigenvector of eigenvalue $1$. Indeed, $c(p_{4,j}\otimes p_{4,j})=p_{4,j}\otimes p_{4,j}$ for $j=0,\,2,\,4$, $c(p_{3,j}\otimes p_{3,j})=p_{3,j}\otimes p_{3,j}$ for $j=1,\,3,\,5$.
\end{proof}

Now we show that Nichols algebras of non-simple indecomposable modules are all infinite-dimensional.
\begin{lem}
Let $V\in\CYD$ be a finite-dimensional module such that $\dim\,\BN(V)<\infty$. Then $\dim\,\BN(W)<\infty$ for all $W\in Soc(V)$ or $W\in Top(V)$.
\end{lem}

\begin{pro}\label{proNicholsindecom}
Let $V\in\CYD$ be a finite-dimensional non-simple indecomposable module. Then $\BN(V)=\infty$.
\end{pro}
\begin{proof}
If $\dim\,V=2$, then by lemma $\ref{lemTwononsimpleindecom}$, we have that $M_l^{+}$, $M_l^{-}$, or $M_l^{\pm}$ for some $l\in Z_6$. In particular,  $Soc(M_l^{+})=\K_{\chi^{l}}, Top(M_l^{+})=\K_{\chi^{l+1}}$, $Soc(M_l^{\pm})=\K_{\chi^{l}}, Top(M_l^{\pm})=\K_{\chi^{l\pm 3}}$, and $Soc(M_l^{-})=\K_{\chi^{l}}, Top(M_l^{-})=\K_{\chi^{l-1}}$. Thus we we must have that $\dim \BN(Top(V))=\infty$ or $\dim\,\BN(Soc(V))=\infty$, which implies $\BN(V)=\infty$.

Now assume that $\dim\,V=d>2$, we prove the claim by induction on d. By Remark $\ref{rmkSocToponedim}$, $Soc(V)$ consists of one-dimensional modules. Let $\widetilde{N}$ be a simple module contained in $Soc(V/Soc(V))$ and $N$ be the corresponding submodule of $V$. Then $\dim\,\widetilde{N}=1$. If the $Soc(V)$ is simple, that is, $V$ is a two-dimensional non-simple indecomposable, then $\dim\,\BN(V)=\infty$. If the $Soc(V)$ is not simple, then it must properly contain a one-dimensional simple module $\K_{\chi^l}$. If $N/\K_{\chi^l}$ is semisimple, $N$ contains an two-dimensional indecomposable module which implies $\dim\BN(N/\K_{\chi^l})=\infty$ and whence $\dim\BN(V)=\infty$. If $N/\K_{\chi^l}$ is not semisimple, it must contain an indecomposable module of dimension less than $d$. By induction, $\dim\BN(N/\K_{\chi^l})=\infty$ and whence $\dim\,\BN(V)=\infty$.
\end{proof}

We have shown that Nichols algebras of non-simple indecomposable modules are all infinite-dimensional. Thus we have the following
\begin{cor}
If $\BN(V)$ is finite-dimensional, then $V$ must be semisimple.
\end{cor}
Next, we analyze the Nichols algebras associated to two-dimensional simple modules.
\begin{lem}
Let $\Lambda\As=\Lambda-\{(1,1),(4,2),(3,1),(2,2),(1,4),(4,5),(2,4),(3,5),\\(4,4),(1,5),(4,1),(1,2)\}$. Then $\dim\,\BN(V_{k,l})=\infty$, for all $(k,l)\in\Lambda\As$.
\end{lem}
\begin{proof}
By Proposition $\ref{probraidsimpletwo}$, the braiding of $V_{k,l}$ for
$(k,l)$ belonging to
$$\{(1,0),(4,3),(3,0), (2,3),(5,0),(0,3),(0,1),(0,2),(3,2),(0,4),(3,4),(0,5)\}$$
has an eigenvector of eigenvalue $1$. Indeed, $c(v_1\otimes v_1)=v_1\otimes v_1$ in the above cases. And $V_{0,1}\As\cong V_{5,2}$,
$V_{0,2}\As\cong V_{5,1}$, $V_{3,2}\As\cong V_{2,1}$, $V_{0,4}\As\cong V_{5,5}$, $V_{3,4}\As\cong V_{2,5}$,
$V_{0,5}\As\cong V_{5,4}$, then by Proposition $\ref{proNicholsdual}$, the claim follows.
\end{proof}

The rest are $V_{1,1}\As\cong V_{4,2}$,
$V_{3,1}\As\cong V_{2,2}$, $V_{1,4}\As\cong V_{4,5}$, $V_{2,4}\As\cong V_{3,5}$, $V_{4,4}\As\cong V_{1,5}$,
$V_{4,1}\As\cong V_{1,2}$. And we will show Nichols algebra $\BN(V)$ is finite-dimensional if V is isomorphic either to $V_{3,1}$, $V_{3,5}$, $V_{2,2}$ or $V_{2,4}$, and describe them in terms of generators and relations.

\begin{pro}\label{proV31}
$\BN(V_{3,1}):=\K\langle v_1, v_2\mid v_1^2=0, v_1v_2-\xi^2v_2v_1=0, v_2^3=0\rangle$. In particular, $\dim\,\BN(V_{3,1})=6$.
\end{pro}
\begin{proof}
In this case, note that $\delta(v_1)=a^5\otimes v_1+(\xi^4-\xi^2)ba^4\otimes v_2$, \\$\delta(v_2)=a^2\otimes v_2+(1+\xi^4)ba\otimes v_1$, and the braiding of $V_{3,1}$ is given by
\begin{align*}
   c(\left[\begin{array}{ccc} v_1\\v_2\end{array}\right]\otimes\left[\begin{array}{ccc} v_1~v_2\end{array}\right])=
   \left[\begin{array}{ccc}
       -v_1\otimes v_1    & \xi^2v_2\otimes v_1+(\xi^2-1)v_1\otimes v_2\\
       v_1\otimes v_2            & \xi^2v_2\otimes v_2+(\xi^4+\xi^2)v_1\otimes v_1
         \end{array}\right].
\end{align*}
Using the braiding of $V_{3,1}$, we have
\begin{align*}
\Delta(v_1^2)&=v_1^2\otimes 1+v_1^2\otimes 1,\\
\Delta(v_1v_2)&=v_1v_2\otimes 1+\xi^2v_1\otimes v_2+\xi^2v_2\otimes v_1+1\otimes v_1v_2,\\
\Delta(v_2v_1)&=v_2v_1\otimes 1+ v_1\otimes v_2+v_2\otimes v_1+1\otimes v_2v_1,\\
\Delta(v_2^2)&=v_2^2\otimes 1+(1+\xi^2)v_2\otimes v_2+(\xi^2+\xi^4)v_1\otimes v_1+1\otimes v_2^2.
\end{align*}
which give us the relations $x^2=0$, and $v_1v_2-\xi^2v_2v_1=0$. And since
\begin{align*}
c(v_2\otimes v_2^2)&=a^2\cdot v_2^2\otimes v_2+(1+\xi^4)ba\cdot v_2^2\otimes v_1\\
&=(a^2\cdot v_2)(a^2\cdot v_2)\otimes v_2+\xi^{-1}\Lam^{-1}(ba\cdot v_2)(ba^4\cdot v_2)\otimes v_2
\\&\quad+(1+\xi^4)(a^5\cdot v_2)(ba^4\cdot v_2)\otimes v_1+(1+\xi^4)(ba^4\cdot v_2)(a^2\cdot v_2)\otimes v_1\\
&=\xi^4v_2^2\otimes v_2+\Lam^{-1}\xi v_1^2\otimes v_2+(1+\xi^4)v_2v_1\otimes v_1+(1+\xi^4)v_1v_2\otimes v_1\\
&=\xi^4v_2^2\otimes v_2+\Lam^{-1}\xi v_1^2\otimes v_2+v_2v_1\otimes v_1,
\end{align*}
we have
\begin{align*}
\Delta(v_2^3)&=(v_2\otimes 1+1\otimes v_2)(v_2^2\otimes 1+(1{+}\xi^2)v_2\otimes v_2+(\xi^2{+}\xi^4)v_1\otimes v_1+1\otimes v_2^2)\\
&=v_2^3\otimes 1+(1+\xi^2)v_2^2\otimes v_2+(\xi^2+\xi^4)v_2v_1\otimes v_1+v_2\otimes v_2^2+\xi^4v_2^2\otimes v_2
\\&\quad+\Lam^{-1}\xi v_1^2\otimes v_2+v_2v_1\otimes v_1+(\xi^2{+}\xi^4)v_2\otimes v_2^2+(1{+}\xi^2{+}2\xi^4)v_1\otimes v_1v_2\\&\quad+(\xi^2{+}\xi^4)v_1\otimes v_2v_1+1\otimes v_2^3\\
&=v_2^3\otimes 1+(1{+}\xi^2{+}\xi^4)v_2^2\otimes v_2+(1{+}\xi^2{+}\xi^4)v_2v_1\otimes v_1+\\&\quad(1+\xi^2+\xi^4)v_2\otimes v_2^2+
(1+\xi^2+\xi^4)(\xi^2+\xi^4)v_1\otimes v_2v_1+1\otimes v_2^3\\
&=v_2^3\otimes 1+1\otimes v_2^3.
\end{align*}
which gives us $x_2^3=0$.
Thus there exists a graded Hopf algebra epimorphism $\pi:R=T(V_{3,1})/I\twoheadrightarrow \BN(V_{3,1})$ in $\CYD$, where I is the ideal generated by the relations $v_1^2=0, v_1v_2-\xi^2v_2v_1=0, v_2^3=0$.

Note that $R^4=0, R^1=V_{3,1}, R^0=\K$. Then by the Poincar\'{e} duality, we have that $\dim\,R^3=\dim\,R^0=1$, and $\dim\,R^2=\dim\,R^1=2$. Since $\dim\,\BN^4(V_{3,1})=0$ and $\pi$ is injective in degree 0 and 1, we have $\dim\,R^3=\dim\,\BN^3(V_{3,1})$ and $\dim\,R^2=\dim\,\BN^2(V_{3,1})$ which implies $\dim\,R=\dim\,\BN(V_{3,1})$. Then the claim follows.
\end{proof}

\begin{pro}
$\BN(V_{3,5}):=\K\langle v_1, v_2\mid v_1^2=0, v_1v_2-\xi^4v_2v_1=0, v_2^3=0\rangle$. In particular, $\dim\,\BN(V_{3,5})=6$.
\end{pro}
\begin{proof}
In this case, $\delta(v_1)=a\otimes v_1+(\xi^4-1)b\otimes v_2$, and $\delta(v_2)=a^4\otimes v_2+(1+\xi^2)ba^3\otimes v_1$. The braiding of $V_{3,5}$ is given by
\begin{align*}
   c(\left[\begin{array}{ccc} v_1\\v_2\end{array}\right]\otimes\left[\begin{array}{ccc} v_1~v_2\end{array}\right])=
   \left[\begin{array}{ccc}
       -v_1\otimes v_1    & \xi^4v_2\otimes v_1+(\xi^4-1)v_1\otimes v_2\\
        v_1\otimes v_2            & \xi^4v_2\otimes v_2+(1+\xi^2) v_1\otimes v_1
         \end{array}\right].
   \end{align*}
From the braiding of $V_{3,5}$, we have
\begin{align*}
\Delta(v_1^2)&=v_1^2\otimes 1+1\otimes v_1^2,\\
\Delta(v_1v_2)&=v_1v_2\otimes 1+ \xi^4v_1\otimes v_2 +\xi^4 v_2\otimes v_1+1\otimes v_1v_2,\\
\Delta(v_2v_1)&=v_2v_1\otimes 1+v_1\otimes v_2+v_2\otimes v_1+1\otimes v_2v_1,\\
\Delta(v_2^2)&=v_2^2\otimes 1+(1+\xi^4)v_2\otimes v_2+(1+\xi^2)v_1\otimes v_1+1\otimes v_2^2,
\end{align*}
which give us the relations $v_1^2=0$ and $v_1v_2-\xi^4v_2v_1=0$. And we also have
\begin{align*}
\Delta(v_2^3)&=(v_2\otimes 1+1\otimes v_2)(v_2^2\otimes 1+(1{+}\xi^4)v_2\otimes v_2+(1{+}\xi^2)v_1\otimes v_1+1\otimes v_2^2)\\
&=v_2^3\otimes 1+(1{+}\xi^4)v_2^2\otimes v_2+(1{+}\xi^2)v_2v_1\otimes v_1+v_2\otimes v_2^2+\xi^2v_2^2\otimes v_2\\&\quad
\Lam^{-1}v_1^2\otimes v_2+\xi^4v_2v_1\otimes v_1+(\xi^4+\xi^2)v_2\otimes v_2^2+\\&\quad(1+\xi^2)(1+\xi^4)v_1\otimes v_1v_2+(1+\xi^2)v_1\otimes v_2v_1+1\otimes v_2^3\\
&=v_2^3\otimes 1+(1+\xi^2+\xi^4)v_2^2\otimes v_2+(1+\xi^2+\xi^4)v_2v_1\otimes v_1\\&\quad+(1+\xi^2+\xi^4)v_2\otimes v_2^2+(1+\xi^2+\xi^4)(1+\xi^2)v_1\otimes v_2v_1\\
&=v_2^3\otimes 1+1\otimes v_2^3,
\end{align*}
since $c(v_2\otimes v_2^2)=\xi^2v_2^2\otimes v_2+\Lam^{-1}v_1^2\otimes v_2+\xi^4 v_2v_1\otimes v_1$. Thus we have the relation $v_2^3=0$.
\end{proof}

\begin{pro}
$\BN(V_{2,2}):=\K\langle v_1, v_2\mid v_1^2-\xi^2v_2^2=0, v_1v_2-v_2v_1=0, v_1^3=0\rangle$. In particular, $\dim\,\BN(V_{2,2})=6$.
\end{pro}
\begin{proof}
In this case, $\delta(v_1)=a^4\otimes v_1+(1-\xi^2)ba^3\otimes v_2$, and $\delta(v_2)=a\otimes v_2+(1+\xi^4)b\otimes v_1$. The braiding of $V_{2,2}$ is given by
\begin{align*}
   c(\left[\begin{array}{ccc} v_1\\v_2\end{array}\right]\otimes\left[\begin{array}{ccc} v_1~v_2\end{array}\right])=
   \left[\begin{array}{ccc}
       \xi^2v_1\otimes v_1    & v_2\otimes v_1+(\xi^2-1)v_1\otimes v_2\\
        \xi^2v_1\otimes v_2            & -v_2\otimes v_2+(1+\xi^4) v_1\otimes v_1
         \end{array}\right].
   \end{align*}
From the braiding of $V_{2,2}$, we have
\begin{align*}
\Delta(v_1^2)&=v_1^2\otimes 1+(1+\xi^2)v_1\otimes v_1+1\otimes v_1^2,\\
\Delta(v_1v_2)&=v_1v_2\otimes 1+ \xi^2v_1\otimes v_2 +v_2\otimes v_1+1\otimes v_1v_2,\\
\Delta(v_2v_1)&=v_2v_1\otimes 1+\xi^2v_1\otimes v_2+v_2\otimes v_1+1\otimes v_2v_1,\\
\Delta(v_2^2)&=v_2^2\otimes 1+(1-\xi)v_1\otimes v_1+1\otimes v_2^2,
\end{align*}
which give us the relations $v_1^2-\xi^2v_2^2=0$ and $v_1v_2-v_2v_1=0$. And we also have
\begin{align*}
\Delta(v_1^3)&=(v_1\otimes 1+1\otimes v_1)(v_1^2\otimes 1+(1+\xi^2)v_1\otimes v_1+1\otimes v_1^2)\\
&=v_1^3\otimes 1{+}(1{+}\xi^2)v_1^2\otimes v_1{+}v_1\otimes v_1^2{+}\xi^4 v_1^2\otimes v_1{+}(\xi^2{+}\xi^4)v_1\otimes v_1^2{+}1\otimes v_1^3\\
&=v_1^3+(1+\xi^2+\xi^4)v_1^2\otimes v_1+(1+\xi^2+\xi^4)v_1\otimes v_1^2+1\otimes v_1^3\\
&=v_1^3\otimes 1+1\otimes v_1^3,
\end{align*}
since $c(v_1\otimes v_1^2)=\xi^4v_1^2\otimes v_1$. Thus we have the relation $v_1^3=0$.
\end{proof}

\begin{pro}
$\BN(V_{2,4}):=\K\langle v_1, v_2\mid v_2^2+\xi v_1^2=0, v_1v_2-v_2v_1=0, v_1^3=0\rangle$. In particular, $\dim\,\BN(V_{2,4})=6$.
\end{pro}
\begin{proof}
In this case, $\delta(v_1)=a^2\otimes v_1+(1-\xi^4)ba\otimes v_2$, and $\delta(v_2)=a^5\otimes v_2+(\xi^4+\xi^2)ba^4\otimes v_1$. The braiding of $V_{2,4}$ is given by
\begin{align*}
   c(\left[\begin{array}{ccc} v_1\\v_2\end{array}\right]\otimes\left[\begin{array}{ccc} v_1~v_2\end{array}\right])=
   \left[\begin{array}{ccc}
       \xi^4v_1\otimes v_1    &  v_2\otimes v_1+(\xi^4-1)v_1\otimes v_2\\
        \xi^4v_1\otimes v_2            & -v_2\otimes v_2+(\xi^4+\xi^2) v_1\otimes v_1
         \end{array}\right].
   \end{align*}
From the braiding of $V_{2,4}$, we have
\begin{align*}
\Delta(v_1^2)&=v_1^2\otimes 1+(1+\xi^4)v_1\otimes v_1+1\otimes v_1^2,\\
\Delta(v_1v_2)&=v_1v_2\otimes 1+ \xi^4v_1\otimes v_2 +v_2\otimes v_1+1\otimes v_1v_2,\\
\Delta(v_2v_1)&=v_2v_1\otimes 1+\xi^4v_1\otimes v_2+v_2\otimes v_1+1\otimes v_2v_1,\\
\Delta(v_2^2)&=v_2^2\otimes 1+(\xi^2+\xi^4)v_1\otimes v_1+1\otimes v_2^2,
\end{align*}
which give us the relations $v_2^2+\xi v_1^2=0$ and $v_1v_2-v_2v_1=0$. And we also have
\begin{align*}
\Delta(v_1^3)&=(v_1\otimes 1+1\otimes v_1)(v_1^2\otimes 1+(1+\xi^4)v_1\otimes v_1+1\otimes v_1^2)\\
&=v_1^3\otimes 1{+}(1{+}\xi^4)v_1^2\otimes v_1{+}v_1\otimes v_1^2{+}\xi^2 v_1^2\otimes v_1{+}(\xi^2{+}\xi^4)v_1\otimes v_1^2{+}1\otimes v_1^3\\
&=v_1^3+(1+\xi^2+\xi^4)v_1^2\otimes v_1+(1+\xi^2+\xi^4)v_1\otimes v_1^2+1\otimes v_1^3\\
&=v_1^3\otimes 1+1\otimes v_1^3,
\end{align*}
since $c(v_1\otimes v_1^2)=\xi^2v_1^2\otimes v_1$. Thus we have the relation $v_1^3=0$.
\end{proof}
\begin{rmk}\label{rmkQues}
Nichols algebras play an important role in the classification of finite-dimensional Hopf algebras, especially of pointed Hopf algebras. However, it is extremely difficult to determine when a Nichols algebra has finite dimension or finite Gelfand-Kirillov dimension or to present by generators and relations. For the Nichols algebras of diagonal type, Heckenberger showed that there exists a close connection to semi-simple Lie algebras, namely he introduced the Weyl groupoid, Weyl equivalence and generalized root system for Nichols algebra. With the help of these concepts he classified all braided vector spaces of diagonal type such that the associated Nichols algebras have the so-called generalized finite root systems $\cite{H09}$. Based on the work of Heckenberger, I.~E.~Angiono determined their explicit relations by generators (\cite{An09}, \cite{An13}, \cite{An15}). But there is no general method, for the Nichols algebras which are not of diagonal type, especially for the Nichols algebras over simple modules. In order to show that a Nichols algebra $\BN(V)$ over a simple module $V$ is infinite-dimensional, we usually try to find a braided subspace or a braided subquotient space $W$ such that $\dim\BN(W)=\infty$ since $\BN(V)$ as an algebra and a coalgebra is completely determined by its braiding. For the category $\CYD$, we have shown that Nichols algebras over non-simple indecomposable modules are infinite-dimensional. Thus, for any object $V$ in $\CYD$, if $\dim \BN(V)<\infty$, the $V$ must be semisimple. Moreover, we show that $\dim\BN(V)=2$ if $V$ is isomorphic to $\K_{\chi^{k}}$ with $k\in\{1,3,5\}$, and $\dim\BN(V)=6$ if $V$ is isomorphic to $V_{3,1}$, $V_{3,5}$, $V_{2,2}$ or $V_{2,4}$.
However, there remain some questions to be solved.
\begin{question}\label{question1}
Determine the braided vector space $W$ such that $\dim\BN(W)<\infty$ where $W$ is isomorphic either to $V_{1,1}$, $V_{4,2}$, $V_{1,4}$, $V_{4,5}$,  $V_{4,4}$, $V_{1,5}$, $V_{4,1}$, $V_{1,2}$ and give an efficient defining set of relations of the Nichols algebras.
\end{question}
Note that after a direct calculation by using the braiding of any $2$-dimensional simple object $W$ in Question $\ref{question1}$, we cannot find a braided subspace $U$ of $W$ such that $\dim\BN(U)=\infty$. Moreover,  Nichols algebra $\BN(W)$ over $W$ as an algebra cannot be isomorphic to some quantum linear space, and the dimension must be bigger than $9$ since one of the relations $v_1^5=0$, $v_1^4=0$, and $v_1^3=0$ must hold in $\BN(W)$.
Thus they cannot produce new Hopf algebras of dimension $72$.

\begin{question}
Determine the braided vector space $V$ such that $\dim\BN(V)<\infty$ where $V$ is a semisimple module, i.e., a direct sum of some simple modules $V_i$ such that $\dim\BN(V_i)<\infty$, and give an efficient defining set of relations of the Nichols algebra $\dim\BN(V)$ $($see \cite{AHS10}\,$)$.
\end{question}
\end{rmk}

\section{Hopf algebras over $\C$}\label{secHopfalgebra}
In this section, we determine all finite-dimensional Hopf algebras
$A$ such that $A_{[0]}\cong \C$ and the corresponding infinitesimal
braiding $V$ is isomorphic either to  $\K_{\chi^{k}}$ with
$k\in\{1,3,5\}$, $V_{3,1}$, $V_{3,5}$, $V_{2,2}$, or $V_{2,4}$. And
we show that there do not exist non-trivial deformations for the
bosonizations of the Nichols algebras associated to the simple
modules as above. As a byproduct, there are $3$ Hopf algebras of dimension $24$
without the Chevalley property given by
$\bigwedge\K_{\chi^{k}}\sharp \C$ for $k\in\{1,3,5\}$. And there are
$4$ Hopf algebras of dimension $72$ without the Chevalley property
given by $\BN(V_{3,1})\sharp\C$, $\BN(V_{3,5})\sharp\C$,
$\BN(V_{2,2})\sharp\C$, $\BN(V_{2,4})\sharp\C$.

First, we show that such Hopf algebra mentioned above is generated in degree one with respect to the standard filtration, i.e., $gr\,A\cong \BN(V)\sharp \C$.
\begin{lem}
Let $A$ be a Hopf algebra of dimension $24$ or $72$ such that $A_{[0]}\cong \C$ and the corresponding infinitesimal braiding $V$ is either the simple modules $\K_{\chi^{k}}$ with $k\in\{1,3,5\}$, or $V_{3,1}$, $V_{3,5}$, $V_{2,2}$, and $V_{2,4}$. Then $gr\,A\cong \BN(V)\sharp \C$. That is, $A$ is generated by the elements of degree one with respect to the standard filtration.
\end{lem}
\begin{proof}
Recall that $H=gr\,A=\oplus_{i\geq 0}A_{[i]}/A_{[i+1]}=R\sharp \C$, where $A_{[0]}\cong\C$ and $R=H^{coA_{[0]}}$. In order to show that $gr\,A\cong \BN(V)\sharp \C$, i.e., $R\cong \BN(V)$, let $S=R\As$ be the graded dual of $R$ and by the duality principle in $\cite[Lemma\,2.4]{AS02}$, $S$ is generated by $S(1)$ since $\Pp(R)=R(1)$. Thus there exists a surjective morphism $S\twoheadrightarrow \BN(W)$ where $W=S(1)$. Thus $S$ is a Nichols algebra if $\Pp(S)=S(1)$, which implies $R$ is a Nichols algebra, i.e., $R=\BN(V)$. To show that $\Pp(S)=S(1)$, it is enough to prove that the relations of $\BN(V)$ also hold in $S$. If $V=\K_{\chi^{k}}=\K[v]/(v^2)$ with some $k\in Z_6$, then $W=\K_{\chi^{l}}=\K[v]/(v^2)$ with some $l\in Z_6$. And if $V$ is isomorphic either to $V_{3,1}$, $V_{3,5}$, $V_{2,2}$ or $V_{2,4}$, then $W$ must be isomorphic either to $V_{3,1}$, $V_{3,5}$, $V_{2,2}$ or $V_{2,4}$, since we have known the dimensions of Nichols algebras over other two-simple modules are bigger than $9$ from the discussion in Remark $\ref{rmkQues}$.

Assume $W=\K_{\chi^{k}}=\K[v]/(v^2)$ with $k\in\{1,3,5\}$ and then $\BN(W)=\bigwedge \K_{\chi^{k}}$ for $k\in\{1,3,5\}$. In such a case, if $v^2\in S$, then $v^2$ is a primitive element and $c(v^2\otimes v^2)=v^2\otimes v^2$. Since as the graded dual of $R, S$ must be finite-dimensional, thus $v^2=0$. Then the claim follows.

Assume that $W=V_{3,1}$, then by Proposition $\ref{proV31}$, we know that as an algebra $\BN(V_{3,1}):=\K\langle v_1, v_2|v_1^2=0, v_1v_2-\xi^2v_2v_1=0, v_2^3=0\rangle$ and the relations of $\BN(V_{3,1})$ are all primitive elements. Thus we need to show that $c(r\otimes r)=r\otimes r$ for $r=v_1^2$, $v_1v_2-\xi^2v_2v_1$ and $v_2^3$. Since
\begin{align*}
\delta(v_1)=a^5\otimes v_1+(\xi^4-\xi^2)ba^4\otimes v_2, \quad \delta(v_2)=a^2\otimes v_2+(1+\xi^4)ba\otimes v_1,
\end{align*}
after a direct computation, we have that
\begin{align*}
\delta(v_1^2)=a^4\otimes v_1^2+(\xi^5-1)ba^3\otimes(v_1v_2-\xi^2v_2v_1),\\
\delta(v_1v_2-\xi^2v_2v_1)=a\otimes (v_1v_2-\xi^2v_2v_1)+\xi^5b\otimes v_1^2,\\
\delta(v_2^3)=1\otimes v_2^3+ba^5\otimes(\xi v_2^2v_1-v_2v_1v_2+\xi^5v_1v_2^2).
\end{align*}
Thus by the braiding of $\CYD$, we have
\begin{align*}
c(v_1^2\otimes v_1^2)=v_1^2\otimes v_1^2,\quad c(v_2^3\otimes v_2^3)=v_2^3\otimes v_2^3,\\
c((v_1v_2-\xi^2v_2v_1)\otimes (v_1v_2-\xi^2v_2v_1))=(v_1v_2-\xi^2v_2v_1)\otimes (v_1v_2-\xi^2v_2v_1).
\end{align*}
Thus the claim follows.
Similarly, the claim follows when $W=V_{3,5}$, $V_{2,2}$, and $V_{2,4}$.
\end{proof}

Next, we shall show that there do not exist non-trivial deformations for the bosonizations of the Nichols algebra associated either to $\K_{\chi^{k}}$ with $k\in\{1,3,5\}$, $V_{3,1}$, $V_{3,5}$, $V_{2,2}$ or $V_{2,4}$.
\begin{lem}\label{lemOnedimNichDeforma}
Let $A$ be a finite-dimensional Hopf algebra over $\C$ such that its infinitesimal braiding $V$ is isomorphic to $\K_{\chi^{k}}$ with $k\in\{1,3,5\}$. Then $A\cong \bigwedge \K_{\chi^{k}}\sharp \C$
\end{lem}
\begin{proof}
Note that $\text{gr}\; A\cong \BN(V)\sharp \C$, where $V$ is isomorphic to $\K_{\chi^{k}}$ with $k\in\{1,3,5\}$. If $v^2\in S$, then $v^2$ is a primitive element and $c(v^2\otimes v^2)=v^2\otimes v^2$. Since as the graded dual of $R, S$ must be finite-dimensional, thus $v^2=0$. Then the claim follows. We prove that the relations also hold in $H$. Indeed, let $\bigwedge \K_{\chi^{k}}=\K[v]/(v^2)$, and $\delta(v)=a^3\otimes v$. Thus
\begin{align*}
\Delta_A(v)&=v\otimes 1+a^3\otimes v,\\
\Delta_A(v^2)&=v^2\otimes 1+(a^3\cdot v+v)\otimes v+1\otimes v^2=v^2\otimes 1+1\otimes v^2.
\end{align*}
But since $A$ is a finite-dimensional Hopf algebra so that $A$ cannot contain any primitive element. Therefore the relation $v^2=0$ must hold in $A$.
\end{proof}

\begin{lem}\label{lemTwodimNichDeforma1}
Let $A$ be a finite-dimensional Hopf algebra over $\C$ such that the infinitesimal braiding $V$ is isomorphic either to $V_{3,1}$ or $V_{3,5}$. Then $A\cong \BN(V)\sharp \C$.
\end{lem}
\begin{proof}
Note that $\text{gr}\; A\cong \BN(V)\sharp \C$, where $V$ is isomorphic either to $V_{3,1}$ or $V_{3,5}$. In order to prove that $A\cong \BN(V)\sharp \C$, we need to show that the homogeneous relations in $\BN(V)$ also hold in $A$.
If $V=V_{3,1}$,  the bosonization $\BN(V_{3,1})\sharp \C$ is generated by $x, y, a, b$ satisfying the relations
\begin{align*}
a^6=1,\quad b^2=0,\quad ba=\xi ab,\quad ax=-xa,\quad bx=-xb,\\
ay+\xi ya=\Lam^{-1}xba^3,\quad by+\xi yb=xa^4,\quad x^2=0,\quad y^3=0,\quad xy-\xi^2yx=0.
\end{align*}
the coalgebra structure is given by
\begin{align*}
\De(a)&=a\otimes a+ \Lam^{-1}b\otimes ba^3,
\De(b)=b\otimes a^4+a\otimes b,\\
\De(x)&=x\otimes 1+a^5\otimes x+(\xi^4-\xi^2)ba^4\otimes y,\\
\De(y)&=y\otimes 1+a^2\otimes y+(1+\xi^4)ba\otimes x.
\end{align*}
We first calculate the following coproducts in $A$:
\begin{align*}
\Delta(xy-\xi^2yx)&=(xy-\xi^2yx)\otimes 1+a\otimes (xy-\xi^2yx)+(1-\xi)b\otimes x^2,\\
\Delta(x^2)&=x^2\otimes 1+a^4\otimes x^2+(1-\xi^2)ba^3\otimes (xy-\xi^2yx),\\
\Delta(y^3)&=y^3\otimes 1+1\otimes y^3+\xi^5\Lam^{-1}(xy-\xi^2yx)ba^4\otimes y+\Lam^{-1}\xi x^2a^2\otimes y \\&\quad +\Lam^{-1}x^2ba\otimes x+\xi^5(xy-\xi^2yx)a^5\otimes x-\xi xa\otimes (xy-\xi^2yx) \\&\quad
-xb\otimes x^2+ba^5\otimes (-yxy+\xi y^2x+\xi^5xy^2).
\end{align*}
Note that $a^4$, $a$ are not group-like elements.
If the relation $x^2=0$ in $\BN(V_{3,1})$ has non-trivial deformations, then $x^2\in A_{[1]}$ and they must be linear combinations of $\{a^i,ba^i, xa^i, ya^i,xba^i,yba^i\}_{i=0}^5$. That is, there exist some elements $\alpha_i,\beta_i,\gamma_i,\lambda_i,\mu_i,\nu_i\in\K$ for $i\in Z_6$ such that
\begin{align*}
x^2=\sum_{i=0}^5\alpha_i a^i+\beta_i ba^i+ \gamma_i xa^i+\lambda_i xba^i+\mu_i ya^i+\nu_i yba^i.
\end{align*}
And we have that
\begin{align*}
x^2a&=\sum_{i=0}^5\alpha_i a^{i+1}+\beta_i ba^{i+1}+ \gamma_i xa^{i+1}+\lambda_i xba^{i+1}+\mu_i ya^{i+1}+\nu_i yba^{i+1},\\
ax^2&=\sum_{i=0}^5\alpha_i a^{i+1}+\beta_i\xi^5 ba^i-\gamma_i xa^{i+1}-\lambda_i\xi^5 xba^{i+1}-\mu_i\xi ya^{i+1}
\\&\quad+\Lam^{-1}\mu_i xba^{3+i}-\nu_i yba^{i+1},
\end{align*}
since $ax=-xa$ and $ay+\xi ya=\Lam^{-1}xba^3$. Then from $ax^2=x^2a$, we have that for all $i\in Z_6$
\begin{align*}
\beta_i=\gamma_i=\lambda_i=\mu_i=\nu_i=0,\quad x^2=\sum_{i=0}^5\alpha_i a^i.
\end{align*}
By the relations $bx=-xb$ and $ba=\xi ab$, we have that $bx^2=x^2b$ and then $\alpha_i=0$ for all $1\leq i\leq 5$. However, from the coproduct of $x^2$, $x^2\neq 1$, thus the relation $x^2=0$ must hold in $A$.
Now we claim that $xy-\xi^2yx=0$ must hold in $A$. Indeed,
Since $xy-\xi^2yx\in A_{[1]}$,  there exist some elements $\alpha_i,\beta_i,\gamma_i,\lambda_i,\mu_i,\nu_i\in\K$ for $i\in Z_6$ such that
\begin{align*}
xy-\xi^2yx=\sum_{i=0}^5\alpha_i a^i+\beta_i ba^i+ \gamma_i xa^i+\lambda_i xba^i+\mu_i ya^i+\nu_i yba^i.
\end{align*}
From the relations $by+\xi yb=xa^4$ and $bx=-xb$, we have that
\begin{align*}
a(xy-\xi^2yx)&=-\Lam^{-1}\xi x^2ba^3+\xi(xy-\xi^2yx)a,\\
b(xy-\xi^2yx)&=-\xi x^2a^4+\xi(xy-\xi^2yx)a.
\end{align*}
Then by the fact that the relation $x^2=0$ must hold in $A$, we have that
\begin{align*}
a(xy-\xi^2yx)=\xi(xy-\xi^2yx)a,\quad
b(xy-\xi^2yx)=\xi(xy-\xi^2yx)b,
\end{align*}
which implies that for all $i\in Z_6$,
\begin{align*}
\alpha_i=\beta_i=\gamma_i=\lambda_i=\mu_i=\nu_i=0.
\end{align*}
Then $\Delta(y^3)=y^3\otimes 1+1\otimes y^3$ and the relation $y^3=0$ must hold in $A$ since $A$ cannot contain primitive elements. Thus $A\cong gr\,A$.\\

If $V=V_{3,5}$,  the bosonization $\BN(V_{3,5})\sharp \C$ is generated by $x, y, a, b$ satisfying the relations
\begin{align*}
a^6=1,\quad b^2=0,\quad ba=\xi ab,\quad ax=-xa,\quad bx=-xb,\\
ay+\xi ya=\Lam^{-1}xba^3,\quad by+\xi yb=xa^4,\quad x^2=0,\quad y^3=0,\quad xy-\xi^4yx=0,
\end{align*}
the coalgebra structure is given by
\begin{align*}
\De(a)&=a\otimes a+ \Lam^{-1}b\otimes ba^3,
\De(b)=b\otimes a^4+a\otimes b,\\
\De(x)&=x\otimes 1+a\otimes x+(\xi^4-1)b\otimes y,\\
\De(y)&=y\otimes 1+a^4\otimes y+(1+\xi^2)ba^3\otimes x.
\end{align*}
Then we have that
\begin{align*}
\Delta(xy-\xi^4yx)&=(xy-\xi^4yx)\otimes 1+a^5\otimes (xy-\xi^4yx)+\xi ba^4\otimes x^2,\\
\Delta(x^2)&=x^2\otimes 1+a^2\otimes x^2-(1+\xi^5)ba\otimes (xy-\xi^4yx),\\
\Delta(y^3)&=y^3\otimes 1+1\otimes y^3+\Lam^{-1}\xi^4(xy-\xi^4yx)b\otimes y+\Lam^{-1}x^2a^4\otimes y \\&\quad
             +(1+\xi^4)(xy-\xi^4yx)a\otimes x+(1+\xi^2)\Lam^{-1}x^2ba^3\otimes x\\&\quad
             +(1+\xi^2)ba^5\otimes(xy^2+\xi^2yxy+\xi^4y^2x)+xa^5\otimes (xy-\xi^4yx)\\&\quad
             +(1+\xi^2)xba^4\otimes x^2.
\end{align*}
Note that $a^2$, $a^5$ are not group-like elements. Similarly to the above proof, since $ax^2=x^2a$ and $bx^2=x^2b$, the relation $x^2=0$ must hold in $A$. After a direct computation, we have that $a(xy-\xi^4yx)=\xi(xy-\xi^4yx)a$ and $b(xy-\xi^4yx)=\xi(xy-\xi^4yx)b$, which imply the relation $xy-\xi^4yx=0$ in $A$. Then $\Delta(y^3)=y^3\otimes 1+1\otimes y^3$ and therefore the relation $y^3=0$ must hold in $A$. Thus $A\cong gr\,A$.
\end{proof}

\begin{lem}\label{lemTwodimNichDeforma2}
Let $A$ be a finite-dimensional Hopf algebra over $\C$ such that the infinitesimal braiding $V$ is isomorphic either to $V_{2,2}$ or $V_{2,4}$. Then $A\cong \BN(V)\sharp \C$.
\end{lem}
\begin{proof}
Note that $\text{gr}\; A\cong \BN(V)\sharp \C$, where $V$ is isomorphic either to $V_{2,2}$ or $V_{2,4}$
if $V=V_{2,2}$,  the bosonization $\BN(V_{2,2})\sharp \C$ is generated by $x, y, a, b$ satisfying the relations
\begin{align*}
a^6=1,\quad b^2=0,\quad ba=\xi ab,\quad ax=\xi^2xa,\quad bx=\xi^2xb,\\
ay+ya=\Lam^{-1}xba^3,\quad by+yb=xa^4,\quad x^2-\xi^2y^2=0,\quad x^3=0,\quad xy-yx=0,
\end{align*}
the coalgebra structure is given by
\begin{align*}
\De(a)&=a\otimes a+ \Lam^{-1}b\otimes ba^3,\quad
\De(b)=b\otimes a^4+a\otimes b,\\
\De(x)&=x\otimes 1+a^4\otimes x+(1-\xi^2)ba^3\otimes y,\\
\De(y)&=y\otimes 1+a\otimes y+(1+\xi^4)b\otimes x.
\end{align*}
Assume that $A$ is a finite-dimensional Hopf algebra such that $gr\,A\cong \BN(V_{2,2})\sharp \C$. We first calculate the following coproducts:
\begin{gather*}
\Delta(x^2-\xi^2y^2)=(x^2-\xi^2y^2)\otimes 1+a^2\otimes (x^2-\xi^2y^2)+\xi^2ba\otimes (xy-yx),\\
\Delta(xy-yx)=(xy-yx)\otimes 1+a^5\otimes(xy-yx)+(1{-}\xi)(\xi^2{-}1)ba^4\otimes (x^2{-}\xi^2y^2),\\
\Delta(x^3)=x^3\otimes 1+1\otimes x^3+(1+\xi)ba^5\otimes (\xi xyx-x^2y+\xi^5 yx^2).
\end{gather*}
Note that $a^2$, $a^5$ are not group-like elements and $ba=\xi ab$. Since $xy-yx\in A_{[1]}$, there exist some elements $\alpha_i,\beta_i,\gamma_i,\lambda_i,\mu_i,\nu_i\in\K$ such that
\begin{align*}
xy-yx=\sum_{i=0}^5\alpha_i a^i+\beta_i ba^i+ \gamma_i xa^i+\lambda_i xba^i+\mu_i ya^i+\nu_i yba^i
\end{align*}
Since $ax=\xi^2xa, ay+ya=\Lam^{-1}xa^4$, we have that $a(xy-yx)=\xi^5(xy-yx)a$, which implies that
\begin{align*}
\alpha_i=\gamma_i=\lambda_i=\mu_i=\nu_i=0,\quad xy-yx=\sum_{i=0}^5\beta_i ba^i.
\end{align*}
However, by the coproducts of $ba^i$ for all $i\in Z_6$ given in Proposition $\ref{proStrucOfC}$ and after a direct computation, we have that $\beta_i=0$ for all $i\in Z_6$.
Then the relation $xy-yx=0$ holds in $A$ and therefore $\Delta(x^3)=x^3\otimes 1 +1\otimes x^3$, which implies $x^3=0$ in $A$.
Similarly, we have that $x^2-\xi^2y^2=0$ must hold in $A$ since
\begin{align*}
a(x^2-\xi^2y^2)=(x^2-\xi^2y^2)a-\xi^2\Lam^{-1}(xy-yx)ba^3=(x^2-\xi^2y^2)a,\\
b(x^2-\xi^2y^2)=(x^2-\xi^2y^2)b-\xi^2(xy-yx)a^4=(x^2-\xi^2y^2)b.
\end{align*}

If $V=V_{2,4}$,  the bosonization $\BN(V_{2,4})\sharp \C$ is generated by $x, y, a, b$ satisfying the relations
\begin{align*}
a^6=1,\quad b^2=0,\quad ba=\xi ab,\quad ax=\xi^2xa,\quad bx=\xi^2xb,\\
ay+ya=\Lam^{-1}xba^3,\quad by+yb=xa^4, y^2+\xi x^2=0, x^3=0, xy-yx=0.
\end{align*}
the coalgebra structure is given by
\begin{align*}
\De(a)&=a\otimes a+ \Lam^{-1}b\otimes ba^3,
\De(b)=b\otimes a^4+a\otimes b,\\
\De(x)&=x\otimes 1+a^2\otimes x+(1+\xi)ba\otimes y,\\
\De(y)&=y\otimes 1+a^5\otimes y+(\xi^2+\xi^4)ba^4\otimes x.
\end{align*}
Following the lifting method in $\cite{AS98}$, we first calculate the following coproducts:
\begin{align*}
\Delta(y^2+\xi x^2)&=(y^2+\xi x^2)\otimes 1+a^4\otimes (y^2+\xi x^2)+\xi^5ba^3\otimes (xy-yx),\\
\Delta(xy-yx)&=(xy-yx)\otimes 1+a\otimes(xy-yx)+(1+\xi^5)b\otimes (y^2+\xi x^2),\\
\Delta(x^3)&=x^3\otimes 1+1\otimes x^3+(1+\xi)ba^5\otimes (yx^2+\xi^2 x^2y+\xi^4 xyx).
\end{align*}
Similarly to the above proof, after a direct calculation, we have that $a(xy-yx)=\xi^5(xy-yx)a$ and, then we have that $xy-yx=\sum_{i=0}^5\beta_i ba^i$. However by the coproducts of $ba^i$ for all $i\in Z_6$ given in Proposition $\ref{proStrucOfC}$ and after a direct computation, we have that $\beta_i=0$ for all $i\in Z_6$. Thus the relation $xy-yx=0$ must hold in $A$ and therefore $\Delta(x^3)=x^3\otimes 1+1\otimes x^3$, which implies $x^3=0$ hold in $A$. Similarly, after a direct computation, we have that $a(x^2+\xi y^2)=(x^2+\xi y^2)a$ and $b(x^2+\xi y^2)=(x^2+\xi y^2)b$, which implies that $x^2+\xi y^2=0$ must hold in $A$.  Thus $A\cong gr\,A$.
\end{proof}

Finally, we end up this section with the following
\begin{thm}
Let $A$ be a finite-dimensional Hopf algebra such that $A_{[0]}=\C$ and its infinitesimal braiding $V$ is isomorphic either to $\K_{\chi^{k}}$ with $k\in\{1,3,5\}$, $V_{3,1}$, $V_{3,5}$, $V_{2,2}$ or $V_{2,4}$. Then $A$ is isomorphic either to
\begin{itemize}
  \item $\bigwedge\K_{\chi^{k}}\sharp \C$, for $k\in\{1,3,5\}$;
  \item $\BN(V_{3,1})\sharp\C$;
  \item $\BN(V_{3,5})\sharp\C$;
  \item $\BN(V_{2,2})\sharp\C$;
  \item $\BN(V_{2,4})\sharp\C$.
\end{itemize}
Moreover, the Hopf algebras $\bigwedge\K_{\chi^{k}}\sharp \C$
without the Chevalley property have dimension $24$, and Hopf
algebras without the Chevalley property given by
$\BN(V_{3,1})\sharp\C$, $\BN(V_{3,5})\sharp\C$,
$\BN(V_{2,2})\sharp\C$, $\BN(V_{2,4})\sharp\C$ have dimension $72$.
\end{thm}

\vskip10pt \centerline{\bf ACKNOWLEDGMENT}

\vskip10pt The paper is supported by the NSFC (Grant No. 11271131).  
N. Hu is indebted to Yinhuo Zhang so much for kindly inviting him to visit 
University of Hasselt and University of Antwerp, Belgium, with the support from these Universities during Dec. 5 to Dec. 18, 2016. Many thanks go to Profs. Yinhuo Zhang and Fred. van Oystaeyen, as well as their colleagues for their hospitalities.

\end{document}